\documentclass[article]{siamart250211}

\usepackage{lipsum}
\usepackage{amsfonts}
\usepackage{graphicx}
\usepackage{epstopdf}
\usepackage{algorithmic}
\ifpdf
  \DeclareGraphicsExtensions{.eps,.pdf,.png,.jpg}
\else
  \DeclareGraphicsExtensions{.eps}
\fi

\newsiamremark{remark}{Remark}
\newsiamremark{hypothesis}{Hypothesis}
\crefname{hypothesis}{Hypothesis}{Hypotheses}
\newsiamthm{claim}{Claim}
\newsiamthm{assum}{Assumption}
\newsiamthm{co}{Corollary}
\newsiamthm{prop}{Proposition}
\newsiamremark{fact}{Fact}
\crefname{fact}{Fact}{Facts}

\headers{Computational Data assimilation}{E. Burman and M. Lu}

\title{Posterior contraction rates of computational methods for Bayesian data assimilation\thanks{Submitted to the editors DATE.}}

\author{Erik Burman\thanks{Department of Mathematics, University College London, London 
  (\email{e.burman@ucl.ac.uk}).}
\and Mingfei Lu\thanks{Department of Mathematics, University College London, London
  (\email{mingfei.lu.21@ucl.ac.uk}).}
}

\usepackage{amsopn}

\ifpdf
\hypersetup{
  pdftitle={Posterior contraction rates of computational methods for Bayesian data assimilation},
  pdfauthor={Erik Burman and Mingfei Lu}
}
\fi
\usepackage{amssymb}

\usepackage{stmaryrd}
\usepackage{amsfonts}
\usepackage{verbatim}
\usepackage{thmtools}
\usepackage{tikz-cd}

\numberwithin{equation}{section}

\begin{document}

\maketitle

\begin{abstract}
    In this paper, we analyze posterior consistency of a Bayesian data assimilation problem under discretization. We prove convergence rates for the discrete posterior to ground truth solution under both conforming discretization and finite element discretization (usually non-conforming). The analysis is based on the coupling of asymptotics between the number of samples and the dimension of discrete spaces. In the finite element discretization, tailor-made discrete priors, instead of the discretization of continuous priors, are used to generate an optimal convergence rate.
\end{abstract}


\section{Introduction}
Bayesian methods for inverse problems subject to partial differential equations have received increasing attention in recent years and enriched deterministic methods through introducing, and quantifying, uncertainty. In these methods, inference is made through possibly noisy observations in addition to the underlying physical laws. Such development extends the deterministic approach and provides a methodology to a larger class of problems where no deterministic method can be applied. For a comprehensive review of such methods, we refer the reader to \cite{stuart2010inverse}. 

A landmark in the recent development of Bayesian inverse problems is the \textit{statistical consistency} (or posterior consistency)—the question of whether the posterior distribution concentrates on the true solution as more data are collected or the noise is vanishing \cite{van2008rates,knapik2011bayesian,agapiou2013posterior,agapiou2014bayesian,florens2016regularizing,monard2021statistical,nickl2024posterior}. Among these references, Nickl \textit{et al}. \cite{abraham2020statistical,giordano2020consistency,nickl2020convergence,nickl2024posterior} have built a framework for posterior contraction for a large class of data assimilation and Bayesian inverse problems with certain stability properties. We also mention \cite{nickl2023bayesian} as a comprehensive material on the topic. We will introduce an ill-posed linear model problem that fits this framework and use it as a baseline for the theoretical developments below.

 Once consistency is established, one can extract point estimates such as the posterior mean via Markov‑chain Monte Carlo (MCMC) sampling or the \textit{maximum a posteriori} (MAP) estimator via optimization. Various types of such computational methods can be found in, e.g., \cite{sanz2024analysis,marzouk2009stochastic,papandreou2023theoretical,stuart2010inverse,nickl2022polynomial}

One of the most fundamental aspects of computational science is discretization. Since computations must be carried out in finite-dimensional spaces, we are always required to approximate infinite-dimensional function spaces with finite-dimensional representations that are suitable for numerical analysis. The dominant strategy in the Bayesian inverse‑problems literature is to truncate an eigen-basis of the relevant Hilbert space. This is an ideal choice since the definition of a Gaussian measure in separable Hilbert spaces is closely related to its basis, and the eigen-basis usually maintains smoothness of the continuous space. Such a discretization is \textit{global} and usually leads to exponential convergence. In \cite{nickl2023bayesian}, the authors proposed a convergence algorithm of this kind for some elliptic inverse problem that maintains posterior consistency. 

On the other hand, in many PDE settings, an explicit eigen-basis may be unavailable. In this case we need to adopt local discretizations such as finite elements. In posterior contraction theory, convergence rates are often generated by the smoothness of the Cameron-Martin space. This, however, may no longer hold in finite element spaces since most finite element spaces are only continuous or continuously differentiable. Therefore, interpolation error must be considered. Recent analyses \cite{sanz2024analysis, papandreou2023theoretical} demonstrate that the finite element posterior does converge to the continuous posterior. However, these results are non-asymptotic in the number of data samples and do not address how the mesh resolution should scale with increasing data size.

We now describe how to incorporate discretization into the posterior contraction analysis. Let $\Pi(\cdot)$ be the prior measure and $D_N$ be some random variable that represents the acquired data. One possible route to obtain posterior contraction rates for such discretization is through the following diagram\\
\[
\begin{tikzcd}
    \Pi_h(\cdot|D_N) \arrow{r}{h\to 0} & \Pi(\cdot|D_N) \arrow{d}{N\to \infty} \\
     & \delta(u^\dagger),
  \end{tikzcd}
\]
where $\Pi_h(\cdot|D_N)$ is some discrete posterior measure and $\delta(u^\dagger)$ is the Dirac measure centred at the ground truth $u^\dagger$. This diagram means that we aim at finding some discrete measure that approximates the continuous posterior measure. Since the continuous posterior measure contracts to the ground truth, the discrete measure also does. References \cite{sanz2024analysis,marzouk2009stochastic,papandreou2023theoretical} give bounds for the first asymptotic $h\to 0$ in this route. However, convergence rates for the second asymptotic $N\to \infty$ may be hard to obtain. When the discrete space is not a subspace of the Cameron–Martin space, the interpolation error of the prior (covariance operator) is hard to control, since they are often defined through 'smoothing operators', such as Mat\'ern type operators, or through a spectral basis. These priors are usually incompatible with finite element spaces. This can lead to a suboptimal convergence (requiring very fine mesh for large $N$) or even no convergence for the discrete measure. Finding a framework where all the asymptotics are respected independently typically necessitates priors that are difficult to implement in practice, for instance defined on Besov spaces, see  \cite{LSS09}.

In this article, we use an alternative route to derive posterior contraction rates for discrete measures. We (in principle) do not need to assign a probability measure to the whole continuous space, but only some finite dimensional space $V_h$. Let $\Pi_h(\cdot)$ be some prior measure on $V_h$, consider the following diagram
\[
\begin{tikzcd}
    \Pi_h(\cdot|D_N) \arrow{r}{N\to \infty} & \delta(u_h^\dagger) \arrow{d}{h\to 0} \\
     & u^\dagger,
  \end{tikzcd}
\]
where $u_h^\dagger$ is some deterministic interpolant of $u^\dagger$ in $V_h$. Because the second step is non‑random, we gain flexibility in constructing priors in $V_h$, and we avoid the assumption that the true solution $u^\dagger$ lies in the Cameron-Martin space (\cref{prior}). The trade-off is a controlled and quantifiable bias introduced by the interpolation process. We show that, with an appropriate coupling of $h$ and $N$, we obtain nearly the same contraction rate as the posterior contraction theory constructed in \cite{nickl2023bayesian}.

\subsection{Notations.} Let $\alpha\ge 0$, $C^\alpha(\Omega)$ ($C_0^\alpha(\Omega)$) be the space of $\alpha$-H\"older continuous functions in $\Omega$ (with compact support). Let $C^\infty(\Omega)$ ($C_0^\infty(\Omega)$) be the space of infinitely differentiable functions (with compact support)  in $\Omega$  Let $H^\alpha(\Omega)$ be the standard Sobolev space with square-integrable derivatives of order less than or equal to $\alpha$ and $H_0^\alpha(\Omega)$ be the completion of $C_0^\infty(\Omega)$ in $H^\alpha(\Omega)$. Let $H^0(\Omega) = L^2(\Omega)$ and $H^{-\alpha}(\Omega)$ be the dual of $H^{\alpha}_0(\Omega)$. For a Hilbert space $H$, we use $\langle\cdot,\cdot\rangle_H$ and $||\cdot||_H$ to represent its inner product and norm. For two Banach spaces $A$ and $B$, we write $A\hookrightarrow B$ as $A$ continuously embedded in $B$. We use $\mathbb{E}^X$, $P_X$ to represent the expectation and probability measure with respect to the random variable $X$. We use $\mathbb{E}(\cdot|\cdot)$ and $P(\cdot|\cdot)$ to represent the conditional expectation and probability, respectively.

\subsection{Contribution and outline of the paper}
In \cref{sec:problem setting and posterior}, we review the posterior contraction theory constructed in \cite{nickl2023bayesian} for Bayesian data assimilation problems and introduce our model problem--the elliptic data assimilation problem using unique continuation, in \cref{sec:elliptic ds}. This is a strongly ill-posed linear inverse problemand challenging already in the deterministic framework. As per our discussion above, it may be difficult to find a covariance operator for this problem that can be optimally interpolated into finite element spaces.

In \cref{sec:conforming space}, we present a posterior contraction result for conforming discretization, that is, the discrete space is assumed to be a subspace of the Cameron-Martin space. We show that the posterior mean (or the MAP estimator) converges to the ground truth by controlling that the dimension of the discrete space is not too high so that it avoids the problem of \textit{overfitting}. A quantitative bound is obtained by setting the dimension of the discrete space $n$ no larger than $N^{\frac{d}{2\alpha+d}}$ asymptotically, whenever $\Omega\subset \mathbb{R}^d$ and the ground truth $u^\dagger\in H^\alpha(\Omega)$. We also discuss the restriction of this method as potential difficulties may arise in categorizing the discrete space.

In \cref{sec: pos cons fem}, we present our main results. Here we use priors that are defined on finite element spaces (that may not be a discretized version of any continuous prior) to obtain the posterior contraction. Although this method does not require a direct application of posterior contraction theory at the continuous level (\cref{consistency}), we obtain nearly the same convergence rate suggested by the continuous theory. The contribution of this development is two-fold. First, it provides a guideline of how to find a discrete posterior measure that contracts to the ground truth solution for Bayesian data assimilation. Second, it extends the continuous posterior contraction theory in the sense that the true solution needs not live in the support of the prior measure, and therefore can be less smooth than what is required by \cref{prior}.

\section{Problem setting and posterior consistency}\label{sec:problem setting and posterior}

\subsection{Problem setting for data assimilation problems with PDE constraints}\label{sec:problem setting}
Let $\Omega\subset \mathbb{R}^d$ be some bounded (possibly space-time) domain. Let $\mathcal{H}_\Omega$ be some Hilbert space of functions defined on $\Omega$. Let $\lambda$ be a probability measure on $\mathbb{R}^d$ with support $\omega\subset \Omega$, where $\omega$ is a precompact subdomain of $\Omega$. Moreover, suppose that $\lambda$ is equivalent to the Lebesgue measure in $\omega$. Let $\{X^{(i)}\}_{i=1}^N$ be a series of i.i.d. random variables with probability law $\lambda$. The function of interest $u\in \mathcal{H}_\Omega $ is measured on the finite set $\mathcal{M}:=\{X^{(i)}\}_{i=1}^N $ with an additive Gaussian noise $\eta$:
\begin{equation}\label{observationmodel}
\begin{array}{cc}
      y = u(X_N)+\eta, & \eta\sim \mathcal{N}(0,\Sigma),
      \end{array}
\end{equation}
where $X_N = (X^{(1)}, X^{(2)}, ..., X^{(N)})^T$, and $\Sigma\in \mathbb{R}^{N\times N}$ is a positive-definite matrix. Our aim is to reconstruct the ground truth $u$, not only in the measurement set $\mathcal{M}$, but the whole domain $\Omega$. We assume that the function of interest is governed by some physical laws described by Partial Differential Equations (PDEs). Let $F$ be some separable Hilbert space. We characterize the PDE model by some linear differential operator $\mathcal{L}:\mathcal{H}_\Omega\to F$ such that 
\begin{equation}\label{solution}
    \mathcal{L}u = f.
\end{equation} 
The PDE system is usually incomplete in terms of missing boundary/initial conditions or parameters, so one cannot reconstruct $u$ solely from $f$. We denote the missing boundary/initial conditions or parameters by $\theta\in \Theta$ for some Hilbert space $\Theta$, and assume there is an affine solution operator $\mathcal{S}:\Theta\to \mathcal{H}_\Omega$, such that
\begin{equation*}
    u = \mathcal{S}(\theta).
\end{equation*} 

There are two approaches to reconstruct the quantity $u$. One is to directly infer $u\in \mathcal{H}_\Omega$ by setting some prior measure in the 'general solution' space $S\subset \mathcal{H}_\Omega$, where $S$ is defined by
\begin{equation*}
    S:= \{u\in \mathcal{H}_\Omega:\mathcal{L}u=f\}.
\end{equation*}
Note that $S$ may not be linear but always affine. This means we can decompose $u = u_1+u_f$ with some known $u_f$ (such $u_f$ can be derived from some well-posed PDE problem with differential operator $\mathcal{L}$), and infer $u_1\in S_0$, where
\begin{equation*}
    S_0:= \{u\in \mathcal{H}_\Omega:\mathcal{L}u=0\},
\end{equation*}
since $S_0$ is usually a subspace of $\mathcal{H}_\Omega$ for a large class of PDE problems. Now through the Bayesian approach \cite{kaipio2006statistical,sanz2018inverse}, formally we have
\begin{equation}\label{bayes}
    \pi_{post} \propto \pi_{like}\cdot \pi_{prior}.
\end{equation}
with
\begin{equation*}
    \pi_{like}(y|u, X_N) \propto \exp(-\frac{1}{2}||y-u|_\omega||^2_{\Sigma^{-1}}),
\end{equation*}
where the norm $||\cdot||_{\Sigma^{-1}}$ is weighted by the covariance matrix, i.e.
\begin{equation*}
    ||z||^2_{\Sigma^{-1}}:= \sum_{i=1}^N\sum_{j=1}^Nz(X^{(i)})\Sigma^{-1}_{ij}z(X^{(j)}).
\end{equation*}
In addition, we assume that the prior measure $\Pi\sim \mathcal{N}(u_0, \mathcal{C}_0)$ is Gaussian and independent with $\{X^{(i)}\}$'s and $\eta$, where $\mathcal{C}_0:S_0\to S_0$ is a trace-class covariance operator and $u_0\in \textup{Im}(\mathcal{C}_0^\frac{1}{2})$. We define the norm on the Cameron--Martin space $E= \textup{Im}(\mathcal{C}_0^\frac{1}{2})$ by
\begin{equation*}
    ||\cdot||^2_{\mathcal{C}_0^{-1}} = \langle\mathcal{C}_0^{-\frac{1}{2}}\cdot, \mathcal{C}_0^{-\frac{1}{2}}\cdot\rangle_{\mathcal{H}_\Omega}
\end{equation*}
Then formally we have
\begin{equation}\label{inferu}
    \pi_{post}(u|y,X_N) \propto \exp(-\frac{1}{2}||y-u|_\omega||^2_{\Sigma^{-1}} - \frac{1}{2}||u - u_0||_{\mathcal{C}_0^{-1}}^2).
\end{equation}

Another possible formulation is to write $u = \mathcal{S}(\theta)$ and set some prior measure $\mathcal{N}(\theta_0,\mathcal{C}_1)$ on $\Theta$. Then the observation model can be written as 
\begin{equation}\label{bayesianobservation}
    y = \mathcal{G}(\theta)(X_N)+\eta,
\end{equation}
where $\mathcal{G}$ is defined by $\mathcal{G}(\theta) := \mathcal{S}(\theta)|_\omega$. The formulation (\ref{bayesianobservation}) is the standard observation model in Bayesian inverse problems (c.f. \cite{stuart2010inverse}). By the same argument as above, we obtain a similar formula for the posterior distribution
\begin{equation}\label{infertheta}
     \pi_{post}(\theta|y,X_N) \propto \exp(-\frac{1}{2}||y-\mathcal{S}(\theta)|_\omega||^2_{\Sigma^{-1}} - \frac{1}{2}||\theta - \theta_0||_{\mathcal{C}_1^{-1}}^2).
\end{equation}

A rigorous way of characterizing  Bayes' rule in the Hilbert space setting is through the prior and posterior measures and representing the likelihood function as a Radon--Nikodym derivative. Let $D_N = \{(X^{(i)}, y^{(i)})\}_{i=1}^N$, then we have 
\begin{equation*}
    \textup{d}\Pi(\cdot|D_N) \propto \pi_{like}\textup{d}\Pi(\cdot).
\end{equation*}

The standard way of extracting information from the posterior measure $\Pi(\cdot|D_N)$ is through computing the posterior mean or the \textit{maximum a posteriori} (MAP) estimator (they are the same in linear/affine problems). Taking formulation (\ref{infertheta}) as an example, such method computes $\theta$ that maximizes the posterior distribution, i.e., let
\begin{equation}\label{penalizedls}
    \mathcal{J}(\theta,y) = \frac{1}{2}||y-\mathcal{S}(\theta)|_\omega||^2_{\Sigma^{-1}}+\frac{1}{2}||\theta-\theta_0||_{\mathcal{C}_1^{-1}}^2.
\end{equation}
We aim at minimizing the cost functional $\mathcal{J}(\theta,y)$ for $\theta\in \text{Im}(\mathcal{C}_0^\frac{1}{2})\subset \Theta$.

\subsection{Posterior consistency}\label{sec:posterior consistency}
In this section, we introduce the general theory of posterior consistency. The theory is established in the recent comprehensive work \cite{nickl2023bayesian} that essentially works for both linear and non-linear Bayesian inverse problems. In this section, we simply write $L^2(\Omega)$ or $H^\alpha(\Omega)$ as $L^2$ or $H^\alpha$ for simplicity. We first give the definition of posterior consistency that we are interested in, which is analogous to \cite[Definition 1]{vollmer2013posterior}.
\begin{definition}
    Let $d$ be a pseudometric in some parameter space $V$ (e.g. $V = \mathcal{H}_\Omega$ or $\Theta$). Consider the observation model with additive noise
    \begin{equation*}
    \begin{array}{cc}
        y = \mathcal{G}(v)(X_N)+\eta, & \eta\sim \mathcal{N}(0,\Sigma).
        \end{array}
    \end{equation*}
    Denote $B^d_r (v)\subset V$ the $d$-ball centered at $v$ with radius $r$, $D_N = \{(X^{(i)}, y^{(i)})\}_{i=1}^N$ the data generated by the ground truth $v^\dagger$, and $P_{D_N}$ the law of $D_N$. \\
    
    (1) If there exists some sequence $\delta_N\downarrow 0$ and $l_n\downarrow 0$ as $N\to \infty$ and some constant M, such that
    \begin{equation}\label{largesamplelimit}
        \lim_{N\to\infty} P_{D_N}(\Pi(B^d_{M\delta_N}(v^\dagger)|D_N)<1-l_n) = 0,
    \end{equation}
    then we say the posterior measure $\Pi(\cdot|D_N)$ is consistent with $v^\dagger$ and the pseudometric $d$ in the \textit{large sample limit}, with convergence rate $\delta_N$.\\
    
    (2) Let $\Sigma = \epsilon \Sigma_0$ for some fixed $\Sigma_0$. If there exists $N = N(\epsilon)$, $\delta_\epsilon\downarrow 0$ and $l_\epsilon\downarrow 0$ as $\epsilon\to 0$ such that 
    \begin{equation}\label{smallnoiselimit}
        \lim_{\epsilon\to 0}P_{D_N}(\Pi(B^d_{M\delta_\epsilon}(v^\dagger)|D_N)<1-l_\epsilon) = 0,
    \end{equation}
    then we say the the posterior measure $\Pi(\cdot|D_N)$ is consistent with $v^\dagger$ and the pseudometric $d$ in the \textit{small noise limit}, with convergence rate $\delta_\epsilon$.
\end{definition}
We now elaborate on this definition. Assume that the ground truth $u^\dagger$ generates the model (\ref{observationmodel}). The posterior consistency in the large sample limit means that, once we obtain more and more samples $\{X^{(i)}\}$, the posterior measure has a high probability (with respect to the law of samples $\{X^{(i)}\}$ and noise $\eta$) to concentrate at $u^\dagger$. The posterior consistency in the small noise limit means that, when the observational noise $\eta$ vanishes, the posterior measure also has a high probability to concentrate at $u^\dagger$, given that we have $N(\epsilon)$ samples. $N(\epsilon)$ usually goes to $\infty$ as $\epsilon\to 0$, as pointwise evaluation should be sufficiently close to $L^2$ evaluation of functions. Since we need to compute the MAP estimator or posterior mean, contractions of such estimators can be categorized in the following sense
\begin{equation*}
    \lim_{N\to\infty} P_{D_N}(||E^\Pi(v|D_N)-v^\dagger||_V>\delta_N) = 0,
\end{equation*}
and 
\begin{equation*}
     \lim_{\epsilon\to 0} P_{D_N}(||E^\Pi(v|D_N)-v^\dagger||_V>\delta_\epsilon) = 0.
\end{equation*}

The behavior of posterior consistency and contraction rates depends critically on the assumptions made about the prior $\Pi$ and the forward operator $\mathcal{G}$. We follow the assumptions in \cite{nickl2023bayesian} to build them for a large class of problems arising in PDE scenarios. Throughout the subsection we make the assumption that the observational noise $\eta$ is i.i.d. standard Gaussian, namely,
\begin{equation}
    \eta\sim \mathcal{N}(0,  I_N),
\end{equation}
where $I_N\in \mathbb{R}^{N\times N}$ is the identity matrix.
\begin{assum}[Forward conditional stability estimates]\label{forward stability}
    Let $V \hookrightarrow L^2$, and $\kappa\ge0$. If for some $M>0$
    then there exists some $U,L>0$ depending on $M$, such that
    \begin{equation}\label{Linftyestimate}
        \sup_{||v||_V\le M}||\mathcal{G}(v)||_{L^\infty}\le U,
    \end{equation}
    and
    \begin{equation}\label{L2estimate}
    \begin{array}{cc}
        ||\mathcal{G}(v_1)-\mathcal{G}(v_2)||_{L^2} \le L||v_1-v_2||_{(H^\kappa)^*}, & \textup{if $||v_1-v_2||_V\le M$}.
        \end{array}
    \end{equation}
\end{assum}

\begin{assum}[Backward conditional stability estimate]\label{backward stability}
  Suppose the ground truth is $v^\dagger$. Let $|\cdot|_V$ be some semi-norm defined on $V$. There exists some constant $\tau\in (0,1)$, for any $M>0$ such that $||v||_V\le M$, there exists some constant $C_M>0$ depending on $M$, such that
  \begin{equation}\label{backwardestimate}
      |v-v^\dagger|_{V} \le C_M||\mathcal{G}(v)-\mathcal{G}(v^\dagger)||_{L^2}^\tau.
  \end{equation}
\end{assum}
The two assumptions are from \cite[Condition 2.1.1, Condition 2.1.4]{nickl2023bayesian}, and we have simplified them by assuming that $V\hookrightarrow L^2$.

\cref{forward stability} and \cref{backward stability} assert forward and backward stability estimates for the forward operator $\mathcal{G}$. The stability is called 'conditional' because estimates (\ref{Linftyestimate}), (\ref{L2estimate}) and (\ref{backwardestimate}) hold under the condition that $||v||_V$ is \textit{a priori} bounded, and stability constants can blow up when $M$ goes to infinity. The idea of conditional stability can be traced back to Tikhonov \cite{tikhonov1943stability}. Conditional stability estimates have been established for a large class of deterministic inverse problems, such as the elliptic Cauchy problem \cite{alessandrini2009stability}. Such estimates typically fall into one of three categories: Lipschitz (e.g., (\ref{L2estimate})), Hölder (e.g., (\ref{backward stability})), or logarithmic. These estimates are essential and usually contribute to the upper bound of convergence rates for statistical and numerical methods.\\

Now we specify the choice of the prior measure $\Pi$ on $V$. Let $\Pi'\sim \mathcal{N}(v_0, \mathcal{C}_0)$ be the base measure on $V$, where $\mathcal{C}_0$ is a trace-class covariance operator and $E = (\textup{Im}(\mathcal{C}_0^\frac{1}{2}), ||\cdot||_{\mathcal{C}_0^{-1}})$.
\begin{assum}\label{prior}
  Let \cref{forward stability} hold with $\kappa\ge0$. Let $\alpha>\frac{d}{2}$, and the base measure $\Pi'\sim \mathcal{N}(0, \mathcal{C}_0)$ with $E = (\textup{Im}(\mathcal{C}_0^\frac{1}{2}), ||\cdot||_{\mathcal{C}_0^{-1}})$ satisfying
    \begin{equation}
    \begin{array}{ccccc}
        E\hookrightarrow H^\alpha_c, & \textup{if $\kappa\ge\frac{1}{2}$} &\textup{or}& E\hookrightarrow H^\alpha, & \textup{if $\kappa<\frac{1}{2}$}.
        \end{array}
    \end{equation}
    The prior measure $\Pi=\Pi_N$ is rescaled according to the law of $v = N^\frac{d}{4\alpha+4\kappa+2d}v'$, where $v'\sim \Pi'$.
\end{assum}
We interpret such choice of rescaled prior. By Bayesian inversion formula, we may write it in the density form
\begin{equation*}
    \textup{d}\Pi_N(\cdot|D_N) \propto \exp(-\frac{1}{2}\sum_{i=1}^N(y^{(i)}-\mathcal{G}(v)(X^{(i)}))^2-\frac{N^\frac{d}{2\alpha+2\kappa+d}}{2}||v||_E^2).
\end{equation*}
Consider the potential with a $\frac{1}{N}$ factor
\begin{equation}\label{potential}
    \mathcal{J}(y,v) = \frac{1}{2 N}\sum_{i=1}^N(y^{(i)}-\mathcal{G}(v)(X^{(i)}))^2+\frac{1}{2N^\frac{2\alpha+2\kappa}{2\alpha+2\kappa+d}}||v||_E^2.
\end{equation}
Then it coincides with the cost function that appears in the least squares formulation with vanishing Tikhonov regularization \cite{engl1996regularization,tikhonov1943stability}. The reason for a $\frac{1}{N}$ factor is to bound the $l^2$-norm by the $L^2$-norm as $N\to \infty$ (at least for smooth functions). Note that from the Bayesian perspective, the prior rescaled with $N^\frac{d}{4\alpha+4\kappa+2d}$ strengthens the regularity of the posterior measure, compared to directly using the base prior (see explanations in \cite[p.32]{nickl2023bayesian}). However, the Tikhonov regularization term still vanishes (although at a lower rate), which reduces the stability when considering discretization.

Now let $v_{min}$ be the minimizer of (\ref{potential}). If the $y^{(i)}$'s are assumed to be exact measurements of $\mathcal{G}(v)(X^{(i)})$'s, then the classical Tikhonov regularization is expected to give the following convergence rate
\begin{equation}
    ||\mathcal{G}(v_{min})-\mathcal{G}(v^\dagger)||_{L^2} = \mathcal{O}_{v^\dagger}({N^{-\frac{\alpha+\kappa}{2\alpha+2\kappa+d}}}).
\end{equation}
The probabilistic adds-on is that such a convergence rate is valid not only for the minimizer of (\ref{potential}), but also for the posterior measure. 
\begin{theorem}\label{consistency}
    Suppose that the forward operator $\mathcal{G}$ satisfies \cref{forward stability} for some $\kappa\ge 0$. Suppose also that the prior $\Pi_N$ distribution satisfies \cref{prior} for some $\alpha>0$ and $E$. Let $\delta_N =N^{-\frac{\alpha+\kappa}{2\alpha+2\kappa+d}}$. Let $M>0$ be large enough. Define
    \begin{equation}\label{V_N}
     V_N:= \{v\in V: v = v_1+v_2, ||v_1||_{(H^\kappa)^*}\le M\delta_N, ||v_2||_{\mathcal{C}_0^{-1}}\le M, ||v||_V\le M\}.
\end{equation}
    If the ground truth $v^\dagger$ satisfies $v^\dagger\in E$, then there exists some $m>0$ and $l_N\downarrow 0$ as $N\to \infty$, such that
    \begin{equation}
        \lim_{N\to\infty}P_{D_N}\left( \Pi_N(v\in V_N: ||\mathcal{G}(v)-\mathcal{G}(v^\dagger)||_{L^2}\le m\delta_N |D_N )\le1-l_N\right)=0.
    \end{equation}
    In addition, if \cref{backward stability} holds for some semi-norm $|\cdot|_V$ and $\mathcal{G}$ with some $\tau>0$, then there exists some $L>0$ and $l'_N\downarrow 0$ as $N\to \infty$, such that
    \begin{equation}
        \lim_{N\to\infty}P_{D_N}\left( \Pi_N(v\in V_N: |v-v^\dagger|_{V}\le L\delta_N^\tau |D_N )\le1-l'_N\right)=0.
    \end{equation}
    Moreover, we have the estimate for the posterior mean
    \begin{equation}
        |\mathbb{E}^\Pi(v|D_N)-v^\dagger|_{V}=\mathcal{O}_{P_{D_N}}(\delta_N^\tau).
    \end{equation}
    This means that the posterior is consistent with respect to $v^\dagger$ in the large sample limit with convergence rate $\delta^\tau_N$.
\end{theorem}
\begin{proof}
    See \cite[Theorem 2.3.1]{nickl2023bayesian}
\end{proof}

\subsection{The elliptic data assimilation problem and stability}\label{sec:elliptic ds}
In this section, we introduce our model problem, that falls into the regime where the assumptions for posterior consistency in \cref{sec:posterior consistency} are valid. Following the general setting in \cref{sec:problem setting}, we let $\mathcal{H}_\Omega  = L^2(\Omega)$, $F = H^{-1}(\Omega)$, $\mathcal{D}(\mathcal{L}) = H^1(\Omega)\cap C^0(\Omega)$ and $\mathcal{L}=-\Delta:\mathcal{D}(\mathcal{L})\to F$. Therefore, the elliptic data assimilation problem is to reconstruct $u\in \mathcal{D}(\mathcal{L})$, from the observation model
\begin{equation}\label{ellipticds}
        \left\{\begin{array}{cc}
         y = u|_\omega +\eta,    & \eta\sim \mathcal{N}(0,\Sigma) \\
          f = -\Delta u + \eta',  & \eta'=0
        \end{array}\right..
    \end{equation}
Note that the boundary condition is not known, so the PDE model is not subject to a direct solution. We show that it satisfies \cref{backward stability} in the following sense.
\begin{lemma}\label{three-balls}
    Let $B\subset\Omega$ be a subdomain of $\Omega$ and the boundary of $B\backslash\omega$ does not touch the boundary of $\Omega$. Let $u\in L^2(\Omega)$. Then there exists a constant $C>0$ and $\tau\in (0,1)$ depending on the geometries, such that
    \begin{equation}\label{L2tbi}
        ||u||_{L^2(B)} \le C(||u||_{L^2(\omega)}+||\Delta u||_{H^{-2}(\Omega)})^\tau ||u||_{L^2{(\Omega)}}^{1-\tau}.
    \end{equation}
    If $u\in H^1(\Omega)$, then there exists a constant $C'>0$ and $\tau'\in (0,1)$ depending on the geometries, such that
    \begin{equation}\label{H1tbi}
        ||u||_{H^1(B)} \le C'(||u||_{L^2(\omega)}+||\Delta u||_{H^{-1}(\Omega)})^{\tau'} ||u||_{H^1{\Omega)}}^{1-\tau'}.
    \end{equation}
\end{lemma}
\begin{proof}
   For the first estimate, see, \cite[Theorem 2.2]{monsuur2024ultra}. The second estimate is adapted from \cite[Theorem]{alessandrini2009stability}, which only shows the bound for $L^2(B)$ norm. However, the bound for $H^1(B)$-norm can be derived by the same process as in \cite{alessandrini2009stability}, by some three-spheres inequality. See, e.g. \cite[Corollary 3]{burman2019unique}.
\end{proof}

We now illustrate the method of using Bayesian inversion to solve (\ref{ellipticds}). As mentioned in \cref{sec:problem setting}, consider the auxiliary problem: find $u_f\in H^1_0(\Omega)$
\begin{equation*}
    \left\{\begin{array}{cc}
      -\Delta u_f =f  &  \textup{in $\Omega$} \\
       u = 0  & \textup{on $\partial\Omega$}
    \end{array}\right..
\end{equation*}
This is a well-posed PDE problem. Assume now we know $u_f$. Let $u^\dagger = u - u_f$. Then $u^\dagger\in H^1(\Omega)$ is a harmonic function. We denote $\mathcal{H}^\alpha(\Omega)$ the space of $H^\alpha(\Omega)$ harmonic functions. The observation model for $u^\dagger$ reads
\begin{equation}\label{u_1model}
\begin{array}{cc}
    y = (u^\dagger+u_f)|_\omega(X_N) + \eta, & \eta\sim \mathcal{N}(0,\Sigma).
\end{array}
\end{equation}
Since $u^\dagger$ is harmonic, it is $C^\infty$ inside $\Omega$. Moreover, since $\bar{\omega}\subset \Omega$, we have $\forall u\in \mathcal{H}^0(\Omega)$
\begin{equation*}
\begin{array}{cc}
    ||u||_{L^\infty(\omega)}\lesssim ||u||_{L^2(\Omega)},\\
||u||_{L^2(\omega)}\le ||u||_{L^2(\Omega)}.
    \end{array}
\end{equation*}
This shows that \cref{forward stability} holds for the forward map with $\kappa=0$. Since \cref{three-balls} shows the forward map also satisfies \cref{backward stability}, if in addition the prior measure satisfies \cref{prior}, then we can apply \cref{consistency} to show that the Bayesian inversion of the elliptic data assimilation problem is posteriori consistent. In fact, we have the following Corollary.
\begin{co}\label{ellipticds consistency}
Let $\Sigma = I_N$ be the identity matrix, $u^\dagger\in \mathcal{H}^\alpha(\Omega)$ with $\alpha>\frac{d}{2}$ be the ground truth in model (\ref{u_1model}), and $\mathcal{C}_0:L^2(\Omega)\to L^2(\Omega)$ be some trace-class covariance operator such that $\textup{Im}(\mathcal{C}_0^{\frac{1}{2}}) \hookrightarrow \mathcal{H}^\alpha(\Omega)$. Let $\Pi_N\sim \mathcal{N}(0, N^\frac{d}{2\alpha+d}\mathcal{C}_0)$. Then the posterior measure $\Pi_N(\cdot|D_N)$ is consistent with $u^\dagger$ and the pseudometric induced by the semi-norm $L^2(B)$ in the large sample limit with convergence rate $\delta_N = N^{-\frac{\alpha\tau}{2\alpha+d}}$, where $B$ and $\tau$ are given by \cref{three-balls}.

\end{co}

\section{Discretization with conforming subspaces}\label{sec:conforming space}
In this section, we discuss the discretization of consistent posteriors for the elliptic data assimilation problem (\ref{ellipticds}). We assume throughout the section that the problem setting is the same as that in \cref{ellipticds consistency}. Since in finite-dimensional spaces, the computation of the posterior mean and the MAP estimator is straightforward, we will mainly focus on the properties of such estimator. 

We give a convergence theorem when the discretization is conforming with respect to the continuous space that supports the prior, i.e., let $V_n\subset E$, where $E= \textup{Im}(\mathcal{C}_0^\frac{1}{2})\hookrightarrow\mathcal{H}^{\alpha}(\Omega)$. In this case, the most natural way is to set the prior in $V_n$, such that
\begin{equation}\label{conforming prior}
    \pi_0(v_n)\propto \exp(-\frac{1}{2}||v_n||_{\mathcal{C}_0^{-1}}^2).
\end{equation}
Such a prior is always well-defined since $V_n$ is finite dimensional, and it is straightforward to compute the MAP estimator in $V_n$ through the following Euler-Lagrange equation
\begin{equation}\label{conformingel}
\begin{array}{cc}
    \langle u^y_n,v_n\rangle_{I_N} + N^\frac{d}{2\alpha+d}\langle u^y_n,v_n\rangle_{\mathcal{C}_0^{-1}}= \langle y,v_n\rangle_{I_N},   & \forall v_n\in V_n .
    \end{array}
\end{equation}

We make a generic assumption that $V_n$ "converges" to $ E$ in the following sense:
\begin{assum}\label{optimal conv projection}
     Assume that $E\hookrightarrow \mathcal{H}^\alpha(\Omega)$ with $\alpha>1$. There exists a projection operator $\mathcal{I}_n:E\to V_n$, such that for any $v\in E$, we have
    \begin{equation*}
    \begin{array}{cc}
        ||v-\mathcal{I}_nv||_{L^2(\Omega)} \le \phi^\alpha(n) ||v||_{E},\\
    \end{array}
    \end{equation*}
    where $\phi(n):\mathbb{N}^+\to \mathbb{R}^+$ is a non-increasing function satisfying $\phi(n) = \mathcal{O}(n^{-\frac{1}{d}})$ asymptotically as $n\to\infty$.
\end{assum}
One could think of $V_n$ as a spectral basis of $\mathcal{H}(\Omega)$ taking the truncation at the $n$-th term. 

\begin{prop}[Galerkin orthogonality]\label{Galerkin orthogonality}
    Let the setting be the same as in \cref{ellipticds consistency}. Let $u^y$ be the posterior mean with the prior $\Pi_N$ and $u^y_n$ be the discrete posterior mean with prior defined in (\ref{conforming prior}). Then the following Galerkin orthogonality holds
    \begin{equation}
        \begin{array}{cc}
         \langle u^y-u^y_n,v_n\rangle_{I_N} + N^\frac{d}{2\alpha+d}\langle u^y-u^y_n,v_n\rangle_{\mathcal{C}_0^{-1}} = 0,    & \forall v_n\in V_n. \\ 
        \end{array}
    \end{equation}
\end{prop}
\begin{proof}
    It can be easily shown by combining the following two identities for $u$ and $u_n$, and the fact that $V_n\subset E$.
    \begin{equation*}
        \begin{array}{cc}
          \langle u^y,v\rangle_{I_N} + N^\frac{d}{2\alpha+d}\langle u^y,v\rangle_{\mathcal{C}_0^{-1}} = \langle y,v\rangle_{I_N},   & \forall v\in E, \\
            \langle u^y_n,v_n\rangle_{I_N} + N^\frac{d}{2\alpha+d}\langle u^y_n,v_n\rangle_{\mathcal{C}_0^{-1}}= \langle y,v_n\rangle_{I_N},   & \forall v_n\in V_h .
        \end{array}
    \end{equation*}
    \[\]
\end{proof}
\begin{remark}
    Note that the above argument makes sense since for linear Bayesian inverse problems with centred Gaussian priors, the posterior mean lives in the Cameron-Martin space. In our case, this means $u^y\in E\hookrightarrow H^\alpha(\Omega)$.
\end{remark}

\begin{theorem}\label{conforming posterior consistency}
    Let the setting be the same as in \cref{ellipticds consistency} and $u^\dagger\in E$ be the ground truth of the elliptic data assimilation problem (\ref{ellipticds}) and $u_n$ satisfy (\ref{conformingel}). Let \cref{optimal conv projection} hold. Assume $l_n$ is an arbitrary sequence such that $l_N\downarrow 0$  and $Nl_N\to\infty$ as $N\to\infty$. If we choose $\phi(n) =\mathcal{O}((Nl_N)^{-\frac{1}{2\alpha+d}})$, then we have
    \begin{equation*}
        ||u^y_n-u^\dagger||_{L^2(B)} = \mathcal{O}_{P_{D_N}}((Nl_N)^{-\frac{\alpha\tau}{2\alpha+d}}),
    \end{equation*}
    where $B$ and $\tau$ are given in \cref{three-balls}.
\end{theorem}

We postpone the proof of \cref{conforming posterior consistency} to \cref{proofsec3}. If we are able to explicitly use $V_n$, then by \cref{conforming posterior consistency}, we could achieve the optimal convergence rate suggested by \cref{consistency}. Typical situations include having a global Hilbert basis for $E$. In this case, \cref{conforming posterior consistency} shows the validity of spectral methods.

Despite its theoretical succinctness, this method has many practical constraints, especially when a global basis is not available. Computing a global basis for non-trivial geometry requires numerical methods and will be subject to discretization error and loss of smoothness. In addition, $V_n$ can never be a finite element space, since we require $v_n\in V_n$ to be pointwise harmonic. There can be some work-around for computing such global basis (which is equivalent to computing fundamental solutions to Poisson's equations). See, for example \cite{barnett2008stability}. Moreover, even if for problems where $V_n$ needs not be a subspace of harmonic functions and thus $V_n$ can be a finite element space, the regularity requirement that $V_n\subset H^\alpha(\Omega)$ is prohibitive due to the complexity of inter-element continuity constraints.

\section{Posterior consistency with finite element priors}\label{sec: pos cons fem}
In this section, we introduce an alternative method to derive posterior contraction rates in some finite element space $V_h$ (usually piecewise polynomials) for our model problem (\ref{ellipticds}), without using the posterior consistency theorem \cref{consistency}. Note that in most cases, we have $V_h\not\subset E$ and thus the prior covariance operator has to be treated carefully in $V_h$. The idea follows the standard Galerkin approximation, i.e., let $\mathcal{I}_h:V\to V_h$ be some projection from the continuous space to the discrete space, we want our discrete solution $u_h$ to satisfy the bound in the following sense:
\begin{equation}\label{galerkin}
    ||u_h-\mathcal{I}_hu^\dagger||_{V} \lesssim ||u^\dagger-\mathcal{I}_hu^\dagger||_V,
\end{equation}
up to some probabilistic conversion. Since the right-hand side converges to $0$ with mesh refinement, we would have posterior consistency directly by the triangle inequality, possibly with some $h = h(N)\downarrow 0$ as $N\to \infty$. The advantage of such method is that the left-hand side of (\ref{galerkin}) only lives in the discrete space, which makes it possible for us to choose a prior defined in the discrete space and avoid difficulties that may arise for probability measures in Hilbert spaces. As we will see, using this method, we can relax \cref{forward stability}, and \cref{prior} by not requiring the prior to live in $H^\alpha(\Omega)$ with $\alpha>\frac{d}{2}$, but only $\alpha>1$. \\

\subsection{Discrete posterior consistency}First, we introduce the discrete setting. Let $\mathcal{T}_h$ be a family of shape regular and quasi-uniform tessellation of $\Omega$ in non-overlapping simplices (for definitions, see, e.g. \cite{ern2021finite}). Let $\omega\subset \Omega$ be the union of a subset $\mathcal{T}_\omega\subset\mathcal{T}_h$. For any two different simplices $K,K'\in\mathcal{T}_h$, $K\cap K'$ consists of either empty set, a common face/edge or vertex. Let the diameter of a simplex $K$ be $h_K$ and its outward normal vector be $n_K$. The global mesh parameter is defined as $h = \max_{K\in\mathcal{T}_h}h_K$. Let $\mathcal{F}$ be the collection of all $d-1$ dimensional facets and $\mathcal{F}_I$ be the collection of all interior $d-1$ dimensional facets. For any facet $F\in\mathcal{F}$, denote $n_F$ the normal vector on $F$, whose direction is arbitrary but fixed. Let $k\ge1$ denote the degree of polynomial approximation and let $\mathbf{P}_k(K)$ be the space of polynomials of degree at most $k$ defined in $K$. Define the $H^1$-conforming finite element space
$$ V^k_h :=\{v_h\in H^1(\Omega):v_h|_K\in \mathbf{P}_k(K), \forall K \in \mathcal{T}_h \}$$
and its subspace with homogeneous boundary conditions
$$W^k_h := V^k_h \cap H^1_0(\Omega).$$
These $H^1$ conforming spaces are well-defined for $k\ge 1$ through e.g., Lagrange finite elements (see \cite[Chapter 19]{ern2021finite}). For $v_h\in V^k_h$, we denote the jump function of its normal gradient across $F\in \mathcal{F}_i$ by 
\begin{equation*}
\llbracket\nabla v_h\cdot n\rrbracket_F :=\nabla v_h\cdot n_1|_{K_1} + \nabla v_h\cdot n_2|_{K_2},
\end{equation*}
where $K_1,K_2\in\mathcal{T}$ are two elements such that $K_1\cap K_2=F$, and $n_1,n_2$ the outward pointing normals with respect to $K_1$ and $K_2$. For all $ v_h \in V^k_h$, define 
\begin{equation*}
    J_h(u_h,v_h) = \sum_{F\in \mathcal{F}_I}\int_Fh\llbracket\nabla u_h\cdot n\rrbracket_F\cdot\llbracket\nabla v_h\cdot n\rrbracket_FdS.
\end{equation*}
We also define the projection of $f\in H^{-1}(\Omega)$ to $W^k_h$ by 
\begin{equation}\label{f_h}
    \langle f_h,w_h\rangle_{L^2(\Omega)} = f(w_h).
\end{equation}
Let $V_h = V_h^k$ be the discrete space with which we will work. We will specify $k$ whenever necessary. Let $W_h = W^1_h$. Define the norms $||\cdot||_{V_h}$ ($||\cdot||_{W_h}$) on $V_h$ ($W_h$)
\begin{equation}\label{discrete norms}
\begin{array}{cc}
      ||v_h||_{V_h}^2 := J_h(v_h,v_h)+ ||h^{(\alpha-1)} v_h||^2_{H^1(\Omega)}, \\
   ||w_h||_{W_h}^2: = J_h(w_h,w_h)+\int_{\partial \Omega}h|\nabla w_h\cdot n|^2dS + ||h w_h||^2_{H^1(\Omega)}. 
\end{array}
\end{equation}
We also define the $W_h^*$-norm of $\Delta v$ with $v\in H^1(\Omega)$ by
\begin{equation}\label{W_h^*}
    ||\Delta v||_{W_h^*} := \sup_{w_h\in W_h\backslash\{0\}}\frac{a(v,w_h)}{||w_h||_{W_h}}.
\end{equation}

Note that since $V_h\subset H^1(\Omega)$, the above definition is valid for all $v_h\in V_h$. Let $X_N = \{X^{(i)}\}_{i=1}^N$ and consider the observational model
\begin{equation}\label{model assumption}
    y = u|_\omega (X_N) + \sigma^2\eta,
\end{equation}
with $\eta\sim \mathcal{N}(0,\Sigma)$ where the eigenvalues of $\Sigma$ locate in $[\sigma^2_{min},\sigma^2_{max}]$ uniformly with respect to $N$.
We now give priors in $V_h$ that depends on the discrete space.
\begin{equation}\label{discrete prior}
    \log\pi_0(u_h) \propto \frac{N}{2\sigma^2}\left(||hu_h||_{V_h}^2+\sum_{K\in \mathcal{T}_h}||h^2(-\Delta u_h-f_h)||^2_{L^2(K)}+ ||h(-\Delta u_h-f)||^2_{W_h^*}\right),
\end{equation}
and
\begin{equation}\label{discrete prior H1}
    \log\pi_1(u_h) \propto \frac{N}{2\sigma^2}\left(||u_h||_{V_h}^2+\sum_{K\in \mathcal{T}_h}||h(-\Delta u_h-f_h)||^2_{L^2(K)}+ ||-\Delta u_h-f||^2_{W_h^*}\right),
\end{equation}
where $f_h$ is defined in (\ref{f_h}). One can see that $\log\pi_0(u_h)$ and $\log\pi_1(u_h)$ only differ by the scale $h^2$. We will use $\pi_0$ for optimal contraction of the posterior under $L^2(B)$ semi-norm, and $\pi_1$ for that under $H^1(B)$ semi-norm (c.f. \cref{three-balls}). Since we are now working in the finite-dimensional space $V_h$, the above definitions indeed give legit Gaussian priors. In this case, the logarithm of posterior density $\log\pi(u_h|y)$ is proportional to the following quadratic form
\begin{equation}\label{costfunctional}
\begin{aligned}
   &\frac{1}{2N}||u_h-y||^2_{\Sigma^{-1}}\\
   +&\frac{1}{2}\left(||h^{\beta}u_h||_{V_h}^2+\sum_{K\in \mathcal{T}_h}||h^{1+\beta}(-\Delta u_h-f_h)||^2_{L^2(K)}+ ||h^\beta(-\Delta u_h-f)||^2_{W_h^*}\right),
    \end{aligned}
\end{equation}
where $\beta = 0,1$. We now interpret the choice of our priors. The $||\cdot||_{V_h}$ norm in the prior works as the penalty of regularity. The difference with imposing a trace-class covariance operator in the continuous case is that, we impose this regularity constraints in the 'weak' sense. Note that for $u\in H^{\frac{3}{2}+\epsilon}$ with any $\epsilon>0$, we have that $J_h(u,u) = 0$. The terms $\sum_{K\in \mathcal{T}_h}||h^{1+\beta}(-\Delta u_h-f_h)||^2_{L^2(K)}$ and $||h^\beta(-\Delta u_h-f)||^2_{W_h^*}$ are used to impose the PDE constraint $-\Delta u = f$ in point-wise and operator sense, which represent penalties on high frequency and low frequency modes, respectively. Since the posterior measure is only defined in the finite element space, it is straightforward to compute the MAP estimator $u_h^y\in V_h$ through the following Euler-Lagrange equation: for all $v_h\in V_h$
\begin{equation}\label{singleEL}
\begin{aligned}
    &\frac{1}{N}\langle u^y_h,v_h\rangle_{\Sigma^{-1}} + s_h(h^\beta u^y_h,h^\beta v_h) + \langle h^\beta \Delta u^y_h,h^\beta\Delta v_h\rangle_{W_h^*}\\ = &\frac{1}{N}\langle y,v_h\rangle_{\Sigma^{-1}}+\sum_{K\in \mathcal{T}_h}\langle -h^{1+\beta}f_h,h^{1+\beta}\Delta v_h\rangle_{L^2(K)}+\langle -h^\beta f,h^{\beta}\Delta v_h\rangle_{W_h^*},
    \end{aligned}
\end{equation}
where $s_h$ is a bilinear form on $V_h$ defined by 
\begin{equation}\label{s_h}
    s_h(u_h,v_h):= \langle u_h,v_h\rangle_{V_h} + \sum_{K\in \mathcal{T}_h}\langle h \Delta u_h,h\Delta v_h\rangle_{L^2(K)}.
\end{equation}
We now present our main results.
\begin{theorem}\label{Eyconv}
    Let $u^\dagger\in H^\alpha(\Omega)\cap C^0(\Omega)$ with $\alpha>1$ be the ground truth solution to the elliptic data assimilation problem (\ref{ellipticds}) and $B$ satisfies the conditions in \cref{three-balls} and $\tau,\tau'\in(0,1)$ be parameters in (\ref{L2tbi}) and (\ref{H1tbi}). Let $V_h = V_h^k$ with $k\ge\min\{\alpha-1,1\}$ and $u_h^y\in V_h$ be the solution to (\ref{singleEL}) with $\beta = 0$ or $1$. 
    
     If $\beta = 1$ and $h = h(N)\sim \mathcal{O}((\frac{\sigma^2}{N})^{-\frac{1}{2\alpha+d}})$, then for $\delta_N = (\frac{\sigma^2}{N})^{-\frac{\alpha}{2\alpha+d}}$, we have
    \begin{equation}\label{EyL2}
        \lim_{N\to\infty}P_{X_N}(\mathbb{E}^y(||u_h^y-u^\dagger||_{L^2(B)}\big|X_N)\lesssim (\delta_N)^\tau) = 1.
    \end{equation}

    Otherwise, if $\beta = 0$ and $h = h(N)\sim \mathcal{O}((\frac{\sigma^2}{N})^{-\frac{1}{2(\alpha-1)+d}})$, then for $\delta'_N = (\frac{\sigma^2}{N})^{-\frac{\alpha-1}{2(\alpha-1)+d}}$, we have
    \begin{equation}\label{EyH1}
        \lim_{N\to\infty}P_{X_N}(\mathbb{E}^y(||u_h^y-u^\dagger||_{H^1(B)}\big|X_N)\lesssim (\delta_N')^{\tau'}) = 1.
    \end{equation}
\end{theorem}
\begin{theorem}\label{discrete posterior consistency}
   Let the setting be the same as in \cref{Eyconv}. Assume $l_n$ is an arbitrary sequence such that $l_N\downarrow 0$  and $Nl_N\to\infty$ as $N\to\infty$. Then
   
   (1) $\beta = 1$: if $h = h(N)=\mathcal{O}( (\frac{\sigma^2}{Nl_N})^{\frac{1}{2\alpha+d}})$, then for $\delta_N = (\frac{\sigma^2}{Nl_N})^{-\frac{\alpha}{2\alpha+d}}$, we have
    \begin{equation}\label{L2posteriorconsistency}
        \lim_{N\to \infty}P_{D_N}(||u_h^y-u^\dagger||_{L^2(B)}\lesssim (\delta_N)^\tau)=1.
    \end{equation}
If $h= h(N)\gg \mathcal{O}( N^{-\frac{1}{2\alpha+d}})$, then 
 \begin{equation}\label{L2convu^y_h}
        \lim_{N\to \infty}P_{D_N}(||u_h^y-u^\dagger||_{L^2(B)}\lesssim h^{\alpha\tau}||u^\dagger||_{H^\alpha(\Omega)})=1.
    \end{equation}

(2) $\beta = 0$: if $h = h(N)\sim \mathcal{O}((\frac{\sigma^2}{Nl_N})^{\frac{1}{2(\alpha-1)+d}})$, then for $\delta'_N = (\frac{\sigma^2}{Nl_N})^{-\frac{\alpha-1}{2(\alpha-1)+d}}$, we have
    \begin{equation}\label{H1posteriorconsistency}
        \lim_{N\to \infty}P_{D_N}( ||u_h^y-u^\dagger||_{H^1(B)}\lesssim (\delta'_N)^{\tau'})=1.
    \end{equation}
If $h=h(N)\gg \mathcal{O}( (\frac{\sigma^2}{N})^{\frac{1}{2(\alpha-1)+d}})$, then 
 \begin{equation}\label{H^1convu^y_h}
        \lim_{N\to \infty}P_{D_N}(||u_h^y-u^\dagger||_{H^1(B)}\lesssim h^{(\alpha-1)\tau'}||u^\dagger||_{H^\alpha(\Omega)})=1.
    \end{equation}
\end{theorem}

We now discuss \cref{Eyconv}, \cref{discrete posterior consistency}, and compare them with the posterior consistency theory in \cref{sec:posterior consistency} and \cref{sec:conforming space}. \cref{Eyconv} and \cref{discrete posterior consistency} do not require the ground truth $u^\dagger$ to live in $\textup{Im}(\mathcal{C}_0^\frac{1}{2})$ for some trace-class covariance operator. This condition usually implies that $u^\dagger\in H^\alpha$ with $\alpha>\frac{d}{2}$. The convergence rate for $L^2(B)$ semi-norm in (\ref{EyL2}) coincides with the convergence rate predicted in \cref{conforming posterior consistency}. Estimate in (\ref{EyH1}) slightly generalizes \cref{forward stability}, since only the following estimate holds
\begin{equation*}
    ||u||_{L^2(\omega)}+||\Delta u||_{H^{-1}(\Omega)} \le C||u||_{H^1(\Omega)},
\end{equation*}
while \cref{forward stability} requires the left-hand side to be bounded by $||u||_{(H^\kappa(\Omega))^*}$ with $\kappa\ge 0$.

Moreover, (\ref{L2convu^y_h}) and (\ref{H^1convu^y_h}) give finite element convergence rates up to stability constants $\tau,\tau'$, which are considered optimal for such ill-posed problems \cite{BNO24}. This means an optimal convergence rate can be obtained with a high probability when the number of samples $N$ is large and/or the level of noise $\sigma^2$ is small.

In addition, the discretization given in \cref{Eyconv} and \cref{discrete posterior consistency} is coarse, since at most we have $h = h(N)\sim \mathcal{O}(N^{-\frac{1}{2\alpha+d}})$. Let $n_{V_h}$ be the global degrees of freedom (or dimension) of $V_h$. Then we have $ n_{V_h}\sim \mathcal{O}(N^{\frac{d}{2\alpha+d}})$, which means the degrees of freedom of the finite element space scales slower than the number of samples, and $n_{V_h}\lesssim N^\frac{1}{2}$ under the usual assumption that the prior $u\in H^\alpha(\Omega)$ with $\alpha>\frac{d}{2}$. If we compute $u^y_h$, the overall computational cost of solving the linear system is at most $\mathcal{O}(N^{\frac{3d}{2\alpha+d}})$, which requires only $\alpha>d$ to make the overall computational cost on par with the sampling cost.

\subsection{Computation of the MAP estimator} 
In this subsection, we introduce the method for computing the MAP estimator. We employ a mixed formulation to avoid direct computation of $||h^\beta(-\Delta u^y_h-f)||_{W_h^*}$. Let $z^y_h \in W_h$ satisfying
\begin{equation}\label{z_h}
    \begin{array}{cc}
     \langle z^y_h,w_h\rangle_{W_h} =-\Delta u^y_h (w_h)-f(w_h) = a(u^y_h,w_h) - f(w_h) ,    & \forall w_h\in W_h. \\
    \end{array}
\end{equation}
Since $z_h\in W_h$ is a Riesz representation of $-\Delta u_h-f\in W_h^*$, we have by the definition of the inner product in dual spaces,
\begin{equation}
     \langle h^\beta (\Delta u^y_h+f),h^\beta\Delta v_h\rangle_{W_h^*} =  a(h^\beta v_h,h^\beta z^y_h).
\end{equation}
Therefore, solving the Euler-Lagrange equation (\ref{singleEL}) is equivalent to finding $(u^y_h,z^y_h)\in V_h\times W_h$, such that for all $(v_h,w_h)\in V_h\times W_h$,
\begin{equation}\label{mixedsys}
    \left\{
\begin{array}{rl}
  \frac{1}{N}\langle u^y_h,v_h\rangle_{\Sigma^{-1}} + s_h(h^\beta u^y_h,h^\beta v_h) \quad & \\[2mm]  + a(h^\beta v_h,h^\beta z^y_h) =  & \sum_{K\in \mathcal{T}_h}\langle -h^{1+\beta}f_h,h^{1+\beta}\Delta v_h\rangle_{L^2(K)}\\[2mm]
  &+\frac{1}{N}\langle y,v_h\rangle_{\Sigma^{-1}} \\[2mm]
     a(u^y_h,w_h) - \langle z^y_h,w_h\rangle_{W_h} =& f(w_h).
\end{array}
    \right.
\end{equation}
Note that the degrees of freedom of the space $W_h$ is smaller than the space $V_h$ since $W_h\subset V_h$ (in fact $W_h$ is only the space of continuous piece-wise affine functions). Therefore, the argument about the computational cost in the last section remains valid and we have the total degree of freedom $n_{V_h}+n_{W_h}\sim \mathcal{O}(N^{\frac{d}{2\alpha+d}})$.
\subsection{Proof of \cref{Eyconv} and \cref{discrete posterior consistency}}
We first give the outline of the proof and some universal notations throughout this subsection. Let $X_N  = (X^{(1)}, X^{(2)}, ..., X^{(N)})^T$ and $D_N = ((X^{(1)},y^{(1)}), (X^{(2)},y^{(2)}), ..., (X^{(N)},y^{(N)}))^T$. Let $y,\eta,\sigma^2,\Sigma, \sigma^2_{min},\sigma^2_{max}$ be from the observation model (\ref{model assumption}). Let $\beta = 0$ or $1$ and $(u^y_h,z_h^y)$ be the solution of (\ref{mixedsys}). Let $u_{\beta} = \mathbb{E}^y(u^y_h\big|X_N)$, which means taking expectation for $u_h^y$ over $y$. Since expectation keeps linearity, we have that $u_\beta$ satisfies: for all $v_h\in V_h$
 \begin{equation}\label{u_hequation}
\begin{aligned}
    &\frac{1}{N}\langle u_\beta,v_h\rangle_{\Sigma^{-1}} + s_h(h^{\beta}u_\beta,h^{\beta}v_h) + \langle h^{\beta} (\Delta u_\beta+f),h^{\beta}\Delta v_h\rangle_{W_h^*}\\ = &\frac{1}{N}\langle u^\dagger,v_h\rangle_{\Sigma^{-1}}+\sum_{K\in \mathcal{T}_h}\langle -h^{1+\beta}f_h,h^{1+\beta}\Delta v_h\rangle_{L^2(K)}.
    \end{aligned}
\end{equation}

Let $\mathcal{I}_h:H^\alpha(\Omega)\to V_h$ be some interpolation operator, $u^\dagger\in H^\alpha(\Omega)$ with $\alpha>1$ be the ground truth and $u_h^\dagger = \mathcal{I}_hu^\dagger$. We have the following inequality
\begin{equation}
    |u_h^y-u_h^\dagger|_V \le |u_h^y-u_\beta|_V+|u_\beta-u_h^\dagger|_V,
\end{equation}
whenever the semi-norm $|\cdot|_V$ makes sense for these functions. One may see $|u_\beta-u_h^\dagger|_V$ as the mean approximation error and $|u_h^y-u_\beta|_V$ as the standard deviation of the estimator. We will first bound the mean error and then the standard deviation. Throughout the proof, we assume that the operator $\mathcal{I}_h$ is the quasi-interpolation operator introduced in \cref{apex:feminequality} (see \cite[Section 6]{ern2017finite} for more details).
\begin{prop}\label{J_hbound}
    For any $v_h\in V_h$, we have
    \begin{equation*}
        J_h(v_h,v_h)\le ||v_h||^2_{H^1(\Omega)}\le h^{-2}||v_h||_{L^2(\Omega)}.
    \end{equation*}
\end{prop}
\begin{proof}
    Direct application of (\ref{tc2}).
\end{proof}

\begin{lemma}\label{three norms}
    Let $u^\dagger\in H^\alpha(\Omega)$. The following estimate holds:
    \begin{equation*}
    \begin{aligned}
       & ||h^\beta u_h^\dagger||_{V_h}^2+\sum_{K\in \mathcal{T}_h}||h^{1+\beta}(-\Delta u_h^\dagger-f_h)||^2_{L^2(K)}+ ||h^\beta (-\Delta u^\dagger_h-f)||^2_{W_h^*}\\
        \le &Ch^{2(\alpha-1+\beta)}||u^\dagger||_{H^\alpha(\Omega)}
        \end{aligned}
    \end{equation*}
\end{lemma}
\begin{proof}
    \textbf{Step 1.} To bound $ ||h^\beta u_h^\dagger||_{V_h}^2$ we start from the identity $ ||h^\beta u_h^\dagger||_{V_h}^2 =  J_h(h^\beta u_h^\dagger,h^\beta u_h^\dagger)+||h^{(\alpha-1+\beta)} u_h^\dagger||_{H^1(\Omega)}^2$. The two terms are bounded by \cite[Proposition 3.3]{BL24}\footnote{Note that, in the reference proposition the Scott-Zhang interpolant is used while here the interpolant is defined in \cref{apex:feminequality}. The reference proposition also holds for the interpolant we use here.} and (\ref{interpolantquasi}) respectively. Therefore
    \begin{equation}\label{V_hbound}
        ||h^\beta u_h^\dagger||_{V_h}^2 \le Ch^{2(\alpha-1+\beta)} ||u^\dagger||^2_{H^\alpha(\Omega)}.
    \end{equation}
    \textbf{Step 2.} For $\sum_{K\in \mathcal{T}_h}||h^{1+\beta}(-\Delta u_h^\dagger-f_h)||^2_{L^2(K)}$, if $\alpha>2$, we have that $f\in H^{\alpha-2}(\Omega)\subset L^2(\Omega)$
    \begin{equation}
        \begin{aligned}
            &\sum_{K\in \mathcal{T}_h}||h^{1+\beta}(-\Delta u_h^\dagger-f_h)||^2_{L^2(K)}\\\le &\sum_{K\in \mathcal{T}_h}||h^{1+\beta}(-\Delta u_h^\dagger-f)||^2_{L^2(K)}+ ||h^{1+\beta}(f-f_h)||_{L^2(\Omega)}^2.
        \end{aligned}
    \end{equation}
    The first term is bounded by $Ch^{2(\alpha-1+\beta)}||u^\dagger||_{H^\alpha(\Omega)}$ by \cite[Proposition 3.3]{BL24}. For the second term, since $f_h$ is the $L^2$-projection of $f$, we have
    \begin{equation*}
        ||h^{1+\beta}(f-f_h)||_{L^2(\Omega)}\le C||h^{\alpha-1+\beta} f||_{H^{\alpha-2}(\Omega)}\le C h^{\alpha-1+\beta} ||u^\dagger||_{H^\alpha(\Omega)}.
    \end{equation*}
    Otherwise, if $\alpha\in [1,2]$, then $f\in H^{\alpha-2}(\Omega)\subset H^{-1}(\Omega)$. We have
    \begin{equation}\label{low regularity residual}
        \sum_{K\in \mathcal{T}_h}||h^{1+\beta}(-\Delta u_h^\dagger-f_h)||^2_{L^2(K)}\le \sum_{K\in \mathcal{T}_h}||h^{1+\beta}\Delta u_h^\dagger||^2_{L^2(K)}+ ||h^{1+\beta}f_h||_{L^2(\Omega)}^2.
    \end{equation}
    Again, the first term in (\ref{low regularity residual}) is bounded by using \cite[Proposition 3.3]{BL24}
    \begin{equation}\label{low regularity laplace}
         \sum_{K\in \mathcal{T}_h}||h^{1+\beta}\Delta u_h^\dagger||^2_{L^2(K)}\le Ch^{2(\alpha-1+\beta)}||u^\dagger||^2_{H^\alpha(\Omega)}.
    \end{equation}
    For the second term, by the definition of $f_h$ in (\ref{f_h}), we have
    \begin{equation*}
    \begin{aligned}
        &\langle h^{1+\beta}f_h,f_h\rangle_{L^2(\Omega)}
        = a(h^{1+\beta}u^\dagger, f_h)\\
        =&\int_\Omega\nabla h^{\beta}(u^\dagger-u_h^\dagger)\cdot \nabla hf_hdx+\int_\Omega\nabla h^{1+\beta}u_h^\dagger\cdot \nabla f_hdx\\
         \le &Ch^{\alpha-1+\beta}||u^\dagger||_{H^\alpha(\Omega)}||f_h||_{L^2(\Omega)} \\
         &+ \sum_{F\in \mathcal{F}_I}\int_F h^{1+\beta}\llbracket \nabla u_h^\dagger\cdot n\rrbracket f_hdS+\langle -h^{1+\beta}\Delta u_h^\dagger,f_h\rangle_{L^2(\Omega)}\\
         \le &C(h^{\alpha-1+\beta}||u^\dagger||_{H^\alpha(\Omega)}+ ||h^\beta u_h^\dagger||_{V_h}+ (\sum_{K\in \mathcal{T}_h}||h^{1+\beta}\Delta u_h^\dagger||^2_{L^2(K)})^\frac{1}{2})||f_h||_{L^2(\Omega)}\\
         \le & Ch^{\alpha-1+\beta}||u^\dagger||_{H^\alpha(\Omega)}||f_h||_{L^2(\Omega)}.
    \end{aligned}
\end{equation*}
    The first inequality uses (\ref{interpolantquasi}), (\ref{inv}) and integration by parts. The second inequality uses the Cauchy-Schwarz inequality, trace inequality (\ref{tc2}) for $f_h$ on $F$ and then inverse inequality (\ref{inv}). The last inequality uses (\ref{low regularity laplace}) and (\ref{V_hbound}) from step 1. By dividing by$||f_h||_{L^2(\Omega)}$ in both sides we obtain
    \begin{equation*}
        ||h^{1+\beta}f_h||_{L^2(\Omega)} \le Ch^{\alpha-1+\beta}||u_h^\dagger||_{H^\alpha(\Omega)}.
    \end{equation*}

    \textbf{Step 3.} Finally we bound $||h^\beta (-\Delta u^\dagger_h-f)||^2_{W_h^*}$. By definition (\ref{W_h^*}), we have for any $w_h\in W_h$
    \begin{equation*}
    \begin{aligned}
        &h^\beta \Delta u^\dagger_h(w_h)-h^\beta f(w_h) \\=& a(h^\beta(u_h^\dagger - u^\dagger), w_h)\\
        = &\sum_{F\in \mathcal{F}_I}\int_Fh^\beta(u_h^\dagger - u^\dagger)\llbracket \nabla w_h\cdot n\rrbracket dS- \sum_{K\in \mathcal{T}_h}\int_Kh^\beta(u_h^\dagger - u^\dagger)\Delta w_h\\
        \le & Ch^\beta||u_h^\dagger-u^\dagger||_{H^1(\Omega)}(J_h(w_h,w_h)^\frac{1}{2}+ (\sum_{K\in \mathcal{T}_h}||h\Delta w_h||^2)^\frac{1}{2})\\
        \le &Ch^{\alpha-1+\beta}||u^\dagger||_{H^\alpha(\Omega)}||w_h||_{W_h}.
        \end{aligned}
    \end{equation*}
    Therefore, we obtain
    \begin{equation*}
        ||h^\beta(-\Delta u^\dagger_h-f)||_{W_h^*}\le Ch^{\alpha-1+\beta} ||u^\dagger||_{H^\alpha(\Omega)}.
    \end{equation*}
\end{proof}

Let $s_h$ be defined as in (\ref{s_h}). We define the triple norm $|||\cdot|||_{\beta}$ on $V_h$ by 
\begin{equation}\label{triplenormhN}
    |||v_h|||_{\beta}^2 = \frac{1}{N}||v_h||^2_{\Sigma^{-1}}+s_h(h^\beta v_h,h^{\beta}v_h)+ ||h^{\beta}\Delta v_h||_{W_h^*}^2.
\end{equation}

\begin{lemma}\label{triplenormbound}
    Let the norm $|||\cdot|||_{\beta}$ be defined as in (\ref{triplenormhN}), and $u_\beta$ be the solution of (\ref{u_hequation}). Then there exists some constant $C>0$ independent of $h$ and $N$, such that for any given samples $\{X^{(i)}\}_{i=1}^N$,
    \begin{equation}
     |||u_\beta-u_h^\dagger|||^2_{\beta}\le C(h^{2(\alpha-1+\beta)}||u^\dagger||_{H^\alpha(\Omega)}^2+\frac{1}{N}||u_h^\dagger-u^\dagger||_{\Sigma^{-1}}^2).
    \end{equation}
\end{lemma}

\begin{proof}
  Let $v_h = u_\beta-u_h^\dagger$, then by (\ref{u_hequation}), we have
\begin{equation*}
\begin{aligned}
    &|||u_\beta-u_h^\dagger|||^2_{\beta} \\= &\frac{1}{N}\langle u^\dagger-u_h^\dagger, u_\beta-u_h^\dagger\rangle_{\Sigma^{-1}}-\sum_{K\in \mathcal{T}_h}\langle h^{1+\beta}(\Delta u_h^\dagger+f_h),h^{1+\beta}\Delta (u_\beta-u_h^\dagger)\rangle_{L^2(K)}\\
    &-J_h(h^{\beta}u_h^\dagger,h^{\beta}(u_\beta-u_h^\dagger))-\langle h^{\alpha-1+\beta} u_h^\dagger, h^{\alpha-1+\beta} (u_\beta-u_h^\dagger)\rangle_{H^1(\Omega)}\\
    &- \langle h^{\beta}(\Delta u_h^\dagger+f), h^{\beta}\Delta (u_\beta-u_h^\dagger)\rangle_{W_h^*} \\
    \le & \frac{1}{2}|||u_\beta-u_h^\dagger|||^2_{\beta} + \frac{1}{2N}|| u^\dagger-u_h^\dagger||^2_{\Sigma^{-1}} +\frac{1}{2}||h^{\beta}u_h^\dagger||_{V_h}^2\\
    &+ \frac{1}{2}||h^{\beta}(-\Delta u^\dagger_h-f)||^2_{W_h^*}+\frac{1}{2}\sum_{K\in \mathcal{T}_h}||h^{1+\beta}(-\Delta u_h^\dagger-f_h)||^2_{L^2(K)}.
    \end{aligned}
\end{equation*}
Therefore by \cref{three norms},
\begin{equation*}
\begin{aligned}
    |||u_\beta-u_h^\dagger|||^2_{\beta} 
    &\le Ch^{2(\alpha-1+\beta)}||u^\dagger||_{H^\alpha(\Omega)} + \frac{1}{N}|| u^\dagger-u_h^\dagger||^2_{\Sigma^{-1}}.
    \end{aligned}
\end{equation*}
This completes the proof.
\end{proof}
\begin{lemma}\label{operatornormsconv}
There exists some constant $C>0$ independent of $h$ and $N$, such that for any set of samples $\{X^{(i)}\}_{i=1}^N$,
    \begin{equation*}
        ||\Delta(u_\beta-u_h^\dagger)||^2_{H^{-1-\beta}(\Omega)} \le C(h^{2(\alpha-1+\beta)} ||u^\dagger||^2_{H^\alpha(\Omega)}+\frac{1}{N}|| u^\dagger-u_h^\dagger||^2_{\Sigma^{-1}}).
    \end{equation*}
\end{lemma}
\begin{proof}
    By the definition of $\Delta (\cdot)$ as an operator, we only need to bound $a(u_\beta-u_h^\dagger,w)$ for $w\in C_0^\infty(\Omega)$. Let $w_h = \mathcal{I}_hw$, we have $a(u_\beta-u_h^\dagger,w) = a(u_\beta-u_h^\dagger,w-w_h)+a(u_\beta-u_h^\dagger,w_h)$. 
    
    \textbf{Step 1.} For the second term, we have
    \begin{equation}\label{residualH1H2bound}
    \begin{aligned}
        &a(u_\beta-u_h^\dagger,w_h)\\
        \le &h^{-\beta}||h^\beta\Delta(u_\beta-u_h^\dagger)||_{W_h^*}||w_h||_{W_h}\\
        \le &h^{-\beta}|||u_\beta-u_h^\dagger|||_{\beta}\left(J_h(w_h,w_h)+\int_{\partial\Omega}h|\nabla w_h\cdot n|^2dS + ||hw_h||_{H^1(\Omega)}^2\right)^\frac{1}{2}.
        \end{aligned}
    \end{equation}
Recall the definition of $||\cdot||_{W_h}$ in (\ref{discrete norms}). By \cref{J_hbound}, (\ref{tc2}) and (\ref{interpolantquasi}), we have
    \begin{equation*}
        ||w_h||_{W_h}\le C||w_h||_{H^1(\Omega)}\le C||w||_{H^1(\Omega)}.
    \end{equation*}
Moreover, since $w\in C_0^\infty(\Omega) $, we have that $\nabla w|_{\partial \Omega}=0$, and
\begin{equation*}
    J_h(w,w_h) = J_h(w,w) = 0.
\end{equation*}
Therefore, by using (\ref{tc1}) and (\ref{interpolantquasi}), we have
\begin{equation*}
\begin{aligned}
    J_h(w_h,w_h)+\int_{\partial\Omega}h|\nabla w_h\cdot n|^2dS &= J_h(w_h-w,w_h-w)+\int_{\partial\Omega}h|\nabla (w_h-w)\cdot n|^2dS\\
    &\le C(||w_h-w||^2_{H^1(\Omega)}+h^2||w_h-w||^2_{H^2(\Omega)})\\
    &\le Ch^2||w||^2_{H^2(\Omega)}.
    \end{aligned}
\end{equation*}

From the above discussion, we obtain the two bounds with respect to $H^1$ and $H^2$ norms.
\begin{equation}\label{w_hboundH1H2}
    \begin{array}{rcl}
     ||w_h||_{W_h} &\le &  C||w||_{H^1(\Omega)},    \\[2mm]
     ||w_h||_{W_h} &\le &Ch||w||_{H^2(\Omega)}.
    \end{array}
\end{equation}
Using (\ref{w_hboundH1H2}) and (\ref{residualH1H2bound}), for $\beta=0,1$, we have
\begin{equation}\label{step1}
\begin{array}{cc}
   a(u_\beta-u_h^\dagger,w_h) \le C|||u_\beta-u_h^\dagger|||_{\beta}||w_h||_{H^{\beta+1}(\Omega)}.
\end{array}
\end{equation}
\textbf{Step 2.} For the first term, through integration by parts, we have
\begin{equation*}
\begin{aligned}
    a(u_\beta-u_h^\dagger,w-w_h) & = \sum_{F\in \mathcal{F}_I}\int_F h^{-(\beta+\frac{1}{2})}(w-w_h)h^\frac{1}{2}\llbracket \nabla h^\beta(u_\beta-u_h^\dagger)\cdot n\rrbracket dS \\
    &- \sum_{K\in \mathcal{T}_h}\langle h^{1+\beta}\Delta (u_\beta-u_h^\dagger), h^{-\beta-1}(w-w_h)\rangle_{L^2(K)}\\
        &\le C\left(\sum_{m=0}^1||h^{-\beta-1+m}(w-w_h)||_{H^m(\Omega)}\right)|||u_\beta-u_h^\dagger|||_{\beta}.
\end{aligned}
\end{equation*}
The inequality uses the Cauchy-Schwarz inequality and then the trace inequality (\ref{tc1}). It then follows by the definition of the triple norm. Now by the approximation inequality (\ref{interpolantquasi}), we obtain that for $\beta=0,1$,
\begin{equation}\label{step2}
\begin{array}{cc}
   a(u_\beta-u_h^\dagger,w-w_h) \le Ch^{-\beta}|||u_\beta-u_h^\dagger|||_{\beta}||w_h||_{H^{1+\beta}(\Omega)}.\\
\end{array}
\end{equation}
Finally, combining (\ref{step1}) and (\ref{step2}), together with \cref{triplenormbound}, we obtain the desired result.
\end{proof}
\begin{lemma}\label{montecarloconv}
    There exists some constant $C>0$ independent of $h$ and $N$, such that
    \begin{equation}\label{probabilityresidual}
    \begin{aligned}
        &P_{D_N}(\frac{1}{N}||u_h^\dagger - u^\dagger||_{\Sigma^{-1}}^2\ge Ch^{2(\alpha-1+\beta)}||u^\dagger||_{H^\alpha(\Omega)}^2) 
        \le \frac{C}{Nh^{2(\alpha-1+\beta)}}.
        \end{aligned}
    \end{equation}
\end{lemma}
Since $u_h^\dagger$ and $u^\dagger$ are both non-random, the random variable $\frac{1}{N}||u_h^\dagger - u^\dagger||_{\Sigma^{-1}}^2$ depends only on $X_N$. Therefore for any Borel measurable set $A \subset \mathbb{R}$, 
\begin{equation*}
P_{D_N}(N^{-1}||u_h^\dagger - u^\dagger||_{\Sigma^{-1}}^2 \in A) = P_{X_N}(N^{-1}||u_h^\dagger - u^\dagger||_{\Sigma^{-1}}^2\in A) .  
\end{equation*}
\begin{proof}
Recall that the smallest eigenvalue of $\Sigma$ is uniformly bounded from below by some $\sigma_{min}^2>0$. Let 
\begin{equation}\label{I_N}
\begin{array}{cc}
     I_N = \frac{1}{ N}\sum_{i=1}^N(u_h^\dagger-u^\dagger)^2(X^{(i)}) .
\end{array}
\end{equation}
we have
\begin{equation*}
    \frac{1}{N}||u_h^\dagger - u^\dagger||_{\Sigma^{-1}}^2\le \frac{1}{\sigma_{min}^2}I_N.
    \end{equation*}
Absorbing the factor $\sigma^2_{min}$ by the constant $C$, the probability in (\ref{probabilityresidual}) is bounded by
\begin{equation}
    P_{D_N}(I_N\ge Ch^{2(\alpha-1+\beta)}||u^\dagger||_{H^\alpha(\Omega)}^2).
\end{equation}
By Chebyshev’s inequality (\ref{eq:cheb}), we obtain
\begin{equation}\label{cheb}
\begin{aligned}
    &P_{D_N}(I_N\ge Ch^{2(\alpha-1+\beta)}||u^\dagger||_{H^\alpha(\Omega)}^2) \\
    =& P_{D_N} (\frac{1}{ \sqrt{N}}\sum_{i=1}^N(u_h^\dagger-u^\dagger)^2(X^{(i)})\ge C h^{2(\alpha-1+\beta)}\sqrt{N}||u^\dagger||_{H^\alpha(\Omega)}^2)\\
    \le &\frac{Var(\sqrt{N}I_N)}{CNh^{4(\alpha-1+\beta)}||u^\dagger||_{H^\alpha(\Omega)}^4}\\
    \end{aligned}
\end{equation}
Since $\lambda$ is equivalent to the Lebesgue measure on $\omega$, there exists some $C>0$, such that
\begin{equation}\label{eqmeasure}
\begin{array}{cc}
    \int_\omega |f|\space\lambda(dx)\le C \int_\omega |f|dx, & \forall f\in L^1(\omega).
    \end{array}
\end{equation}
Using (\ref{eqmeasure}), Jensen's inequality, H\"older's inequality, the approximation inequality (\ref{interpolantquasi}), and the stability inequality (\ref{interpolantLinfty}), we obtain
\begin{equation}\label{varbound}
\begin{aligned}
    Var(\sqrt{N}I_N) &= \frac{1}{N}\mathbb{E}^{X_N}\left(\sum_{i=1}^N(u_h^\dagger-u^\dagger)^2(X^{(i)})- N\int_\omega (u_h^\dagger-u^\dagger)^2\lambda(dx)\right)^2\\
    &\le \frac{1}{N}\sum_{i=1}^N\mathbb{E}^{X^{(i)}}\left((u_h^\dagger-u^\dagger)^2(X^{(i)})-\int_\omega (u_h^\dagger-u^\dagger)^2\lambda(dx)\right)^2\\
    &\le \mathbb{E}^{X^{(i)}}\left( (u_h^\dagger-u^\dagger)^4(X^{(i)})\right)- (\int_\omega (u_h^\dagger-u^\dagger)^2\lambda(dx))^2\\
    &\le C||u_h^\dagger - u^\dagger||_{L^4(\omega)}^4\\
    &\le C||u_h^\dagger - u^\dagger||_{L^2(\Omega)}^2||u_h^\dagger - u^\dagger||_{L^\infty(\Omega)}^2\\
    &\le Ch^{2\alpha}||u^\dagger||_{H^\alpha(\Omega)}^2||u^\dagger||_{L^\infty(\Omega)}^2
    \end{aligned}
\end{equation}

Therefore, using (\ref{cheb}) and the fact that $u^\dagger\in H^\alpha(\Omega)\cap C^0(\Omega)$ and $\beta\le 1$, we obtain
\begin{equation*}
     P_{D_N}(I_N\ge Ch^{2(\alpha-1+\beta)}||u||_{H^\alpha(\Omega)}^2) \le \frac{C}{Nh^{2\alpha+4(\beta-1)}}\le \frac{C}{Nh^{2(\alpha-1+\beta)}}.
\end{equation*}
This ends the proof.
\end{proof}
\begin{co}\label{bound of residuals}
    Let $u_\beta$ solve (\ref{u_hequation}) with $\beta = 0$ or $1$ and $D_N = \{(X^{(i)},y^{(i)})\}_{i=1}^N$. There exists some constant $C>0$ independent of $h$ and $N$, such that
     \begin{equation*}
     \begin{aligned}
          &P_{D_N}(||\Delta(u_\beta-u_h^\dagger)||^2_{H^{-1-\beta}(\Omega)} \ge Ch^{2(\alpha-1+\beta)}||u^\dagger||_{H^\alpha(\Omega)}^2)
          \le \frac{C}{Nh^{2(\alpha-1+\beta)}}.
          \end{aligned}
     \end{equation*}
\end{co}
\begin{proof}
    It can be shown by combining \cref{operatornormsconv} and \cref{montecarloconv}.
\end{proof}

\cref{bound of residuals} provides a convergence rate for $\Delta(u_\beta-u_h^\dagger)$ with respect to $h$ under $H^{-1}$ and $H^{-2}$ norms ($\beta = 0,1$). In view of \cref{three-balls}, we also need to give the convergence rate for $u_\beta-u_h^\dagger$ under the $L^2(\omega)$ semi-norm.
\begin{prop}
    For any $v_h\in V_h$, there exists a constant $C>0$ independent of $h$, such that
    \begin{equation}\label{Linftybound}
        ||v_h||^2_{L^\infty(\Omega)} \le Ch^{-d}||v_h||_{L^2(\Omega)}^2.
    \end{equation}
\end{prop}
\begin{proof}
We only need to show that (\ref{Linftybound}) holds for a basis of $V_h$. We begin by selecting a set of Lagrange basis functions $\{\phi_1, \phi_2,...,\phi_n\}$, constructed via a geometric mapping (see, e.g. \cite[(28.13)]{ern2021finite2})  from the standard equispaced Lagrange polynomials defined on a reference simplex $\hat{K}$. Clearly on the reference simplex, there exists a uniform $C>0$, such that for any such Lagrange polynomial $\phi$ we have
\begin{equation*}
    ||\phi||^2_{L^\infty(\hat{K})} \le C||\phi||_{L^2(\hat{K})}^2,
\end{equation*}
by norm equivalence. For any $K\in \mathcal{T}_h$ with diameter $h$, the change of variable leads to the factor $h^{-d}$, and thus (\ref{Linftybound}) holds for every basis function $\phi_j$.
\end{proof}
\begin{lemma}[Empirical process in the finite element space]\label{empirical bound}\\
    There exists some $c\in (0,1)$, and $m>0$ small enough, all of which are independent of $h$ and $N$, such that there exists some Borel measurable set $\Theta_N\subset \{X_N\in \mathbb{R}^N:\forall v_h\in V_h, \frac{1}{N}\sum_{i=1}^Nv_h^2(X^{(i)})\ge c ||v_h||^2_{L^2(\omega)}\}$ satisfying
    \begin{equation}
        P_{D_N}(\Theta_N)\ge 1-h^{-d}\exp(-mNh^{d}).
    \end{equation}
\end{lemma}
\begin{proof}
    The lemma follows from the matrix Chernoff bound (\cref{thm:chernoff}) and the estimate in (\ref{Linftybound}). Let $\lambda$ be the Lebesgue measure in $\omega$. First, we choose a set of Lagrange basis functions for $V_h$ with unit $L^2$ norms, denoted by $\{\phi_1,...,\phi_{n_1}\}\cup\{\psi_{1},...\psi_{n_2}\}$, where $\{\phi_i\}$ denote base functions whose degrees of freedom are associated with nodal points in $\bar{\omega}$ and $\{\psi_j\}$ are associated with nodal points in $\bar{\Omega}\backslash\bar{\omega}$. This is valid since we assumed that $\omega$ is the union of a sub-family of $\mathcal{T}_h$. Clearly, we can decompose any $v_h\in V_h$ uniquely as $v_h = v_1+v_2$, such that $v_2$ is a linear combination of $\{\psi_j\}$. Therefore $||v_2||_{L^2(\omega)} = 0$. Let $V_1 = \textup{span}\{\phi_1,...,\phi_{n_1}\}$. Then
    \begin{equation*}
    \begin{aligned}
         \{X_N:\forall v_h\in V_h, \frac{1}{N}\sum_{i=1}^Nv_h^2(X^{(i)})\ge c ||v_h||^2_{L^2(\omega)}\}\\
         =  \{X_N:\forall v_1\in V_1, \frac{1}{N}\sum_{i=1}^Nv_1^2(X^{(i)})\ge c ||v_1||^2_{L^2(\omega)}\}.
         \end{aligned}
    \end{equation*}
    This is because for any $v_2\in \textup{span}\{\psi_{1},...\psi_{n_2}\}$, the inequality is trivial. Let $A$ be the Gram matrix with respect to $\{\phi_1,...,\phi_{n_1}\}$ and the $L^2$ inner product. Since these basis functions are assumed to have a unit $L^2$ norm, the spectrum of $A$ is bounded in some interval $[\mu_{min},\mu_{max}]$ uniformly with respect to $h$, where $\mu_{min}$ $(\mu_{max})$ is the smallest (largest) eigenvalue of $A$. Let $\Phi(X^{(k)}) = (\phi_1(X^{(k)}), \phi_2(X^{(k)}), ...,\phi_{n_1}(X^{(k)}))^T$, $Y_k = \frac{1}{N}\Phi(X^{(k)})\Phi(X^{(k)})^T$ and $A_N = \sum_{k=1}^NY_k$. Clearly we have that 
    \begin{equation*}
        \begin{array}{cc}
            \mathbb{E}^{D_N}(A_N) = A,   \\[2mm]
            ||Y_k||^2_{op} \le \frac{1}{N^2}\sum_{i=1}^{n_1}\phi_i^2(X^{(k)})\lesssim  \frac{1}{N^2h^d}\max_{i=\{1,...,n_1\}}||\phi_i||_{L^\infty(\omega)}^2\lesssim \frac{1}{N^2h^{2d}}& a.e. \;X_N.
        \end{array}
    \end{equation*}
Therefore, the largest eigenvalue of $Y_k$ is bounded by $\frac{1}{Nh^{d}}$. In the second line, we use the boundedness of the $L^2$ norm of $\{\phi_i\}$ and (\ref{Linftybound}). Let $\lambda_{min}$ be the smallest eigenvalue of $A_N$. By \cref{thm:chernoff}, we obtain that there exists some $C>0$, such that for any $\epsilon\in (0,1)$,
\begin{equation*}
    \begin{aligned}
        P_{D_N}(\lambda_{min}\le \epsilon\mu_{min})
        \le Ch^{-d}\exp\left(Nh^d(\epsilon\log\epsilon-(1-\epsilon)\mu_{min})\right).
    \end{aligned}
\end{equation*}
We can choose $\epsilon $ small enough, so that $\epsilon\log\epsilon-(1-\epsilon)\mu_{min}<0$. Then there exists some $m=m(\epsilon)>0$ small enough that absorbs the constant $C$, such that
\begin{equation*}
    P_{D_N}(\lambda_{min}\ge \epsilon\mu_{min}) \le 1- h^{-d}\exp(-mNh^{d}).
\end{equation*}
  Moreover, since we have the set inclusion
    \begin{equation*}
     \{X_N:\lambda_{min}\ge \epsilon \mu_{min}\}\subset\{X_N:\forall v_h\in V_h, \frac{1}{N}\sum_{i=1}^Nv_h^2(X^{(i)})\ge \epsilon\frac{\mu_{min}}{\mu_{max}}||v_h||^2_{L^2(\omega)}\},
    \end{equation*}
Clearly  $\{X_N:\lambda_{min}\ge \epsilon \mu_{min}\}$ is Borel measurable since $\lambda_{min}$ is a composition of continuous functions with respect to $X_N$. We conclude that $\Theta_N = \{X_N:\lambda_{min}\ge \epsilon \mu_{min}\}$ and the bound in the lemma holds for $c\le \epsilon\frac{\mu_{min}}{\mu_{max}}$.
\end{proof}

\begin{lemma}\label{L2errorbound}
  There exists some constants $m>0$ small enough, and $C, M$ large enough, all of which are independent of $h$ and $N$, such that
    \begin{equation}
    \begin{aligned}
        &P_{D_N}(||u_h^\dagger - u_\beta||_{L^2(\omega)}^2\ge Mh^{2(\alpha-1+\beta)}||u^\dagger||_{H^\alpha(\Omega)}^2) \\
        \le &\frac{C}{Nh^{2(\alpha-1+\beta)}}+ h^{-d}\exp(-mNh^{d})
        \end{aligned}
    \end{equation}
\end{lemma}
\begin{proof}
Let $\Theta_N \subset \{\forall v_h\in V_h, \frac{1}{N}\sum_{i=1}^Nv_h^2(X^{(i)})\ge c ||v_h||^2_{L^2(\omega)}\}$ be the Borel set from in \cref{empirical bound}. Let $C$ be the same as in \cref{montecarloconv}, and $M = \frac{C\sigma_{max}^2}{c}$. Then if $\xi(h,u^\dagger) = h^{2(\alpha-1+\beta)}||u^\dagger||_{H^\alpha(\Omega)}^2$,
 \begin{equation*}
     \begin{aligned}
        &P_{D_N}\big(||u_\beta-u_h^\dagger||_{L^2(\omega)}^2\le M\,\xi(h,u^\dagger)\big)\\
         \ge & P_{D_N}\big({N}^{-1}||u_h^\dagger - u^\dagger||_{\Sigma^{-1}}^2\le C\,\xi(h,u^\dagger), \frac{\sigma_{max}^2}{N}||u_\beta-u_h^\dagger||^2_{\Sigma^{-1}}\ge c ||u_\beta-u_h^\dagger||_{L^2(\omega)}^2\big)\\
         \ge & P_{D_N}\big({N}^{-1}||u_h^\dagger - u^\dagger||_{\Sigma^{-1}}^2\ge C\,\xi(h,u^\dagger), \Theta_N\big)\\
         \ge & 1-\frac{C}{N\,h^{2(\alpha-1+\beta)}}-h^{-d}\exp(-mNh^d),
     \end{aligned}
 \end{equation*}
 where the first and second inequalities are set inclusions, and the last inequality applies \cref{montecarloconv} and \cref{empirical bound}.
 \[\]
\end{proof}

Now we define the distance between $v,v'\in H^1(\Omega)$
\begin{equation}\label{d_h}
    d_\beta(v,v') :=\left(||v-v'||_{L^2(\omega)}^2+||\Delta(v-v')||^2_{H^{-1-\beta}(\Omega)} \right)^\frac{1}{2}.
\end{equation}

\begin{co}\label{residual norm conv}
    If $h = h(N)$ satisfies 
     \begin{equation*}
     \lim_{N\to\infty}Nh^{2(\alpha-1+\beta)}=\infty,
     \end{equation*}
     then there exists some constant $C>0$, such that
     \begin{equation*}
         \lim_{N\to\infty}P_{D_N}(d_\beta(u_\beta,u_h^\dagger)\le Ch^{\alpha-1+\beta}||u^\dagger||_{H^\alpha(\Omega)}) = 1
     \end{equation*}
\end{co}
\begin{proof}
    Direct application of \cref{bound of residuals} and \cref{L2errorbound}.
\end{proof}

\begin{lemma}\label{variancebound}
 Let $d_\beta$ be defined in (\ref{d_h}). Then there exists some $m>0$ small enough and $C>0$, such that
    \begin{equation}
    \begin{array}{cc}
        P_{X_N}\big(\mathbb{E}^y(d_\beta(u_h^y,u_\beta)\big|X_N)\le \frac{C\sigma}{\sqrt{Nh^d}}\big)\ge 1-h^{-d}\exp(-mNh^d).
        \end{array}
    \end{equation}
\end{lemma}
\begin{proof}
    Let $e_h^y = u_h^y-u_\beta$. Then $e_h^y$ satisfies the following equation
    \begin{equation}\label{e_h^yequation}
\begin{aligned}
    &\frac{1}{N}\langle e_h,v_h\rangle_{\Sigma^{-1}} + s_h(e_h,v_h) + \langle h \Delta e_h,h\Delta v_h\rangle_{W_h^*} = \frac{1}{N}\langle \eta,v_h\rangle_{\Sigma^{-1}}.
    \end{aligned}
\end{equation}
  We denote $\{\phi_1,...\phi_{n}\}$ a Lagrange basis of $V_h$ where only the first $n_1$ functions do not vanish in $\omega$. Let $\mathbf{\Phi}\in \mathbb{R}^{N\times n_1}$ such that $\mathbf{\Phi}_{ij} = \phi_j(X^{(i)})$. Let $\mathbf{K}\in\mathbb{R}^{n\times n}$ such that 
  \begin{equation*}
      \mathbf{K}_{i,j} = s_h(\phi_j,\phi_i) + \langle h \Delta \phi_j,h\Delta \phi_i\rangle_{W_h^*}.
  \end{equation*}
Let $\mathbf{W}\in \mathbb{R}^{n\times n}$ such that
\begin{equation*}
    \mathbf{W} = \left(\begin{array}{cc}
      \Phi^T\Sigma^{-1}\Phi   & 0 \\
        0 & 0
    \end{array}\right).
\end{equation*}
Define $\mathbf{L}\in \mathbb{R}^n$ by $\mathbf{L} = (\eta(X^{(1)}), \eta(X^{(2)}),...,\eta(X^{(N)}))^T$ and $\mathbf{M}\in \mathbb{R}^n$ by $\mathbf{M} = (\mathbf{\Phi}^T \Sigma^{-1} \mathbf{L},\mathbf{0})^T$. The matrix representation of (\ref{e_h^yequation}) reads
\begin{equation*}
    (\frac{1}{N}\mathbf{W}+\mathbf{K})E^y_h = \frac{1}{N}\mathbf{M},
\end{equation*}
where $E^y_h$ represents the coefficient of $e_h^y$ with respect to the basis $\{\phi_1,...\phi_{n}\}$. By the definition of $\mathbf{M}$, we have $\mathbb{E}^y(\mathbf{M}\mathbf{M}^T\big|X_N) = \sigma^2\mathbf{W}$. By the definition of the triple norm (c.f. (\ref{triplenormhN})), there exists some $C>0$, such that
\begin{equation}\label{|||e_h^y|||}
    \begin{aligned}
        \mathbb{E}^y(|||e_h^y|||^2\space\big|X_N) &= \mathbb{E}^y\left((E_h^y)^T(\frac{1}{N}\mathbf{W}+\mathbf{K})E^y_h\big|X_N\right)\\
        & = \frac{\sigma^2}{N} \textup{tr}\left((\frac{1}{N}\mathbf{W}+\mathbf{K})^{-1}\frac{1}{N}\mathbf{W}\right)\\
        & = \frac{\sigma^2}{N} \textup{tr}\left(I - (\frac{1}{N}\mathbf{W}+\mathbf{K})^{-1}\mathbf{K}\right)\\
        &\le \frac{C\sigma^2}{Nh^d}
    \end{aligned}
\end{equation}
where the last line uses the fact that $n\le Ch^{-d}$. This shows that the bound for $|||e_h^y|||$ is valid for almost every $X_N$. Let $\Theta_N \subset \{\forall v_h\in V_h, \frac{1}{N}\sum_{i=1}^Nv_h^2(X^{(i)})\le c ||v_h||^2_{L^2(\omega)}\}$ chosen as in \cref{empirical bound}. By using the same technique as in the proof of \cref{L2errorbound}, there exists some $C$ large enough and $m>0$ small enough, such that
\begin{equation*}
    P_{X_N}\big(\mathbb{E}^y(d(u_h^y,u_\beta)\big|X_N)\le \frac{C\sigma}{\sqrt{Nh^d}}\big)\ge 1- h^{-d}\exp(-mNh^d).
\end{equation*}
This concludes the proof.
\end{proof}

\begin{proof}[Proof of \cref{Eyconv} and \cref{discrete posterior consistency}]
\textbf{Step 1.} We prove \cref{Eyconv}. First we decompose $u_h^y-u^\dagger$ as $(u_h^y-u_\beta)+(u_\beta-u_h^\dagger)+(u_h^\dagger-u^\dagger)$. Since $u_\beta = \mathbb{E}^y(u_h^y\big|X_N)$, the latter two terms in the decomposition do not depend on $y$. Therefore, by the triangle inequality
  \begin{equation}\label{dbetaconv}
      \begin{aligned}
          &P_{X_N}\left(\mathbb{E}^y(d_\beta(u_h^y,u^\dagger)\big|X_N)\le Ch^{\alpha-1+\beta}||u^\dagger||_{H^\alpha(\Omega)}+ \frac{C\sigma}{\sqrt{Nh^d}}\right)\\
          \ge & P_{X_N}\left(\mathbb{E}^y(d_\beta(u_h^y,u_\beta)\big|X_N)\le \frac{C\sigma}{\sqrt{Nh^d}}, d_\beta(u_\beta,u_h^\dagger)\le Ch^{\alpha-1+\beta}||u^\dagger||_{H^\alpha(\Omega)}\right),\\
      \end{aligned}
  \end{equation}
  since $d_\beta(u_h^\dagger,u^\dagger)\le Ch^{\alpha-1+\beta}||u||_{H^\alpha(\Omega)}$. By the definition of $d_\beta$ (c.f. (\ref{d_h})), to extract an optimal rate, we need to equate the two rates $\frac{\sigma^2}{Nh^d}$ and $ h^{2(\alpha-1+\beta)}||u^\dagger||_{H^\alpha(\Omega)}$. We treat $||u^\dagger||_{H^\alpha(\Omega)}$ as a constant, in which case we have $h(N) = \mathcal{O}((\frac{\sigma^2}{N})^{\frac{1}{2(\alpha-1+\beta)+d}})$, and thus the condition $\lim_{N\to\infty}Nh^{2(\alpha-1+\beta)} = \infty$ is satisfied. Therefore by \cref{residual norm conv} and \cref{variancebound}, we have
\begin{equation*}
    \lim_{N\to\infty}P_{X_N}\left(\mathbb{E}^y(d_\beta(u_h^y,u^\dagger)\big|X_N)\le C\Bigl(\frac{\sigma^2}{N}\Bigr)^{\frac{\alpha-1+\beta}{2(\alpha-1+\beta)+d}}\right)=1.
\end{equation*}
Now by \cref{three-balls}, the desired convergence in \cref{Eyconv} is a consequence of the following statements:
\begin{equation*}
    \begin{array}{cc}
 \big(\mathbb{E}^y(d_1(u_h^y,u^\dagger)\big|X_N)\le C(\frac{\sigma^2}{N})^{\frac{\alpha}{2\alpha+d}}\big)  \Rightarrow     \big(\mathbb{E}^y(||u_h^y-u^\dagger||_{L^2(B)}\big|X_N)\le C'(\frac{\sigma^2}{N})^{\frac{\alpha\tau}{2\alpha+d}}\big)   \\
      \big(\mathbb{E}^y(d_0(u_h^y,u^\dagger)\big|X_N)\le C(\frac{\sigma^2}{N})^{\frac{\alpha-1}{2\alpha+d}}\big)  \Rightarrow     \big(\mathbb{E}^y(||u_h^y-u^\dagger||_{H^1(B)}\big|X_N)\le C'(\frac{\sigma^2}{N})^{\frac{(\alpha-1)\tau'}{2(\alpha-1)+d}}\big) .
    \end{array}
\end{equation*}

\textbf{Step 2.} We prove \cref{discrete posterior consistency}. Consider the second line in (\ref{dbetaconv}). According to (\ref{|||e_h^y|||}), $\mathbb{E}^y(|||u_h^y-u_\beta|||^2\big|X_N)\le \frac{C\sigma^2}{Nh^d}$ for $a.e.\, X_N$. By Chebyshev's inequality (\ref{eq:cheb}), for any sequence $l_N$ that satisfies $l_N\to 0$ and $Nl_N\to \infty$, we have
\begin{equation*}
\begin{array}{cc}
    P_{y}(|||u_h^y-u_\beta|||_\beta\ge \frac{C\sigma}{\sqrt{Nl_Nh^d}}\big|X_N)\le \frac{Nl_Nh^d\mathbb{E}^y(|||u_h^y-u_\beta|||^2_\beta\space|X_N)}{C\sigma^2}\le l_N, & \textup{for a.e. }X_N.
    \end{array}
\end{equation*}
Since $l_N$ does not depend on $X_N$, we have that 
\begin{equation*}
    P_{D_N}(|||u_h^y-u_\beta|||_\beta\ge \frac{C\sigma}{\sqrt{Nl_Nh^d}})\le l_N.
\end{equation*}
  Then by repeating the last paragraph in the proof of \cref{variancebound}, we obtain
\begin{equation*}
    P_{D_N}(d_\beta(u_h^y,u_\beta)\le \frac{C\sigma}{\sqrt{Nl_Nh^d}})\ge 1-h^{-d}\exp(-mNl_Nh^d)-l_N.
\end{equation*}
The optimal choice for the discretization parameter $h$ is to equate $h^{2(\alpha-1+\beta)}||u^\dagger||_{H^\alpha(\Omega)}$ and $\frac{\sigma^2}{Nl_Nh^{d}}$. This leads to $h\sim \mathcal{O}((\frac{\sigma^2}{Nl_N})^\frac{1}{2\alpha+d})$ for the $L^2(B)$ semi-norm, whereas $h\sim \mathcal{O}((\frac{\sigma^2}{Nl_N})^\frac{1}{2(\alpha-1)+d})$ for the $H^1(B)$ semi-norm. Note that $\lim_{N\to\infty}Nh^{2(\alpha-1+\beta)} = \infty$ still holds since here $h$ is larger than it was in Step 1. We complete the proofs of (\ref{L2posteriorconsistency}) and (\ref{H1posteriorconsistency}) by the same arguments closing Step 1. Finally, when $h=h(N)$ vanishes asymptotically slower than $(\frac{\sigma^2}{N})^{\frac{1}{2(\alpha-1+\beta)+d}}$, which means $h\gg \mathcal{O}((\frac{\sigma^2}{N})^{\frac{1}{2(\alpha-1+\beta)+d}})$, there must exist some intermediate rate between $h(N)$ and $(\frac{\sigma^2}{N})^{\frac{1}{2(\alpha-1+\beta)+d}}$. Therefore $l_N$ can be chosen accordingly such that $h\gg \mathcal{O}((\frac{\sigma^2}{Nl_N})^{\frac{1}{2(\alpha-1+\beta)+d}})$. In this case $h^\alpha||u^\dagger||_{H^\alpha(\Omega)}$ dominates $\frac{\sigma}{\sqrt{Nl_Nh^d}}$, and thus (\ref{L2convu^y_h}) and (\ref{H^1convu^y_h}) hold by \cref{three-balls}.
\end{proof}

\section*{Acknowledgement}
E.B. was supported by the EPSRC grants EP/T033126/1 and
  EP/V050400/1. For the purpose of open access, the author has applied a Creative Commons Attribution (CC BY) licence to any Author Accepted Manuscript version arising.

\appendix

\section{Finite element interpolation and inequalities}\label{apex:feminequality}
We present here some basic interpolations and finite element inequalities that has been used repeatedly in the article. Let $\Omega\subset \mathbb{R}^d$ be a polygonal/polyhedral domain and let $\mathcal{T}_h$ be a quasi-uniform family of triangulations and $h$ be the maximum diameter of $\mathcal{T}_h$. Denote $\mathbf{P}_k(K)$ as the space of polynomials in $K$ with orders at most $k$.

Now we give some useful inequalities and approximation results (see, for example, \cite[Section 1.4.3]{di2011mathematical}).

(1) Continuous trace inequality.
\begin{equation}\label{tc1}
    ||v||_{L^2(\partial K)}\le C(h^{-\frac{1}{2}}||v||_{L^2(K)}+h^{\frac{1}{2}}||\nabla v||_{L^2(K)}), \;\; \forall v\in H^1(K).
\end{equation}

(2) Discrete trace inequality.
\begin{equation}\label{tc2}
    ||\nabla v_h \cdot n||_{L^2(\partial K)}\le Ch^{-\frac{1}{2}}||\nabla v_h||_{L^2(K)}, \;\; \forall v_h\in \mathbf{P}_k(K).
\end{equation}

(3) Inverse inequality.
\begin{equation}\label{inv}
    ||\nabla v_h||_{L^2(K)}\le Ch^{-1}||v_h||_{L^2(K)}, \;\; \forall v_h\in\mathbf{P}_k(K).
\end{equation}\\

We now introduce the quasi-interpolation operator \cite[Section 6]{ern2017finite}, defined by $\pi^k_{q}: L^1(\Omega) \to V_h^k$. It keeps homogeneous boundary condition and has the optimal approximation property:
\begin{equation}\label{interpolantquasi}
(\sum_{K\in\mathcal{T}_h}||v-\pi_{q}^kv||^2_{H^m(K)})^\frac{1}{2} \le Ch^{s-m}||v||_{H^s(\Omega)}, 
\end{equation}
whenever $v\in H^{s}(\Omega)$ with $\frac{1}{2}<s\le k+1$ and $m\in \{0:[s]\}$, and 
\begin{equation}\label{interpolantLinfty}
    ||\pi_q^k v||_{L^\infty(\Omega)}\le C||v||_{L^\infty(\Omega)}
\end{equation}
if $v\in L^\infty(\Omega)$.

\section{Probability inequalities}
In this appendix we collect some standard probability inequalities. Let the triplet $(\Omega, \mathcal{F},P)$ be some probability space and $X:\Omega\to \mathbb{R}$ be some random variable, then we have Chebyshev’s inequality (see, e.g. \cite{durrett2019probability}):
\begin{equation}\label{eq:cheb}
    \begin{array}{cc}
       P(|X|>a)\le \frac{E(X^2)}{a^2},  & \forall a>0. \\
    \end{array}
\end{equation}

We also give the matrix Chernoff's bound (see, e.g., \cite[Theorem 5.1.1]{tropp2015introduction}).
\begin{theorem}\label{thm:chernoff}
    Let $\{X_k\}$ be a series of i.i.d. $d$-dimensional symmetric matrices. Assume also their eigenvalues are in $[0,L_d]$, where $L_d>0$ may depend on the dimension $d$. Let $Y = \sum X_i$. Denote $\lambda_{min}(\cdot)$ ($\lambda_{max}(\cdot)$) the smallest (largest) eigenvalue of some symmetric matrix and 
    \begin{equation*}
        \begin{array}{cc}
        \mu_{min} = \lambda_{min}(\mathbb{E}(Y))     \\
         \mu_{max} = \lambda_{max}(\mathbb{E}(Y)).
        \end{array}
    \end{equation*}
    Then,
    \begin{equation*}
    \begin{array}{cc}
        P\left(\lambda_{min}(Y)\le (1-\epsilon)\mu_{min}\right) \le d(\frac{e^{-\epsilon}}{(1-\epsilon)^{1-\epsilon}})^{\frac{\mu_{min}}{L_d}}, & \forall \epsilon\in [0,1)\\
        P\left(\lambda_{max}(Y)\ge (1+\epsilon)\mu_{max}\right) \le d(\frac{e^{\epsilon}}{(1+\epsilon)^{1+\epsilon}})^{\frac{\mu_{max}}{L_d}}, & \forall \epsilon\ge 0.
    \end{array}
    \end{equation*}
    \end{theorem}
\section{Proof of \cref{conforming posterior consistency}}\label{proofsec3}
In this appendix, we give the proof of \cref{conforming posterior consistency}. Let $u_n = \mathbb{E}^y(u^y_n\big|X_N)\in V_n$. Then $u_n$ satisfies 
\begin{equation}\label{u_nequation}
    \begin{array}{cc}
    \langle u_n,v_n\rangle_{I_N} + N^\frac{d}{2\alpha+d}\langle u_n,v_n\rangle_{\mathcal{C}_0^{-1}}= \langle u^\dagger,v_n\rangle_{I_N},   & \forall v_n\in V_n .
    \end{array}
\end{equation}
Let $u_n^\dagger = \mathcal{I}_nu^\dagger$, where $\mathcal{I}_n$ satisfies \cref{optimal conv projection} with $\phi^\alpha(n) = \mathcal{O}((Nl_N)^{-\frac{\alpha}{2\alpha+d}})$. We thus decompose $u^y_n-u^\dagger = (u^y_n-u_n) + (u_n-u_n^\dagger)+(u_n^\dagger-u^\dagger)$. We elaborate the first difference as the variance, the second the mean approximation error, and the third the interpolation error.  Let $\delta_N = N^{-\frac{\alpha}{2\alpha+d}}$. Note that $\delta_N\ll(Nl_N)^{-\frac{\alpha}{2\alpha+d}}$. For the third term, by \cref{optimal conv projection}
\begin{equation}
    ||u_n^\dagger-u^\dagger||_{L^2(B)}\lesssim (Nl_N)^{-\frac{\alpha}{2\alpha+d}}||u^\dagger||_{H^\alpha(\Omega)}.
\end{equation}
We only need bound the first two terms. \cref{conforming posterior consistency} will be a direct consequence of \cref{meanerroru_n} and \cref{varianceu_n}

\begin{lemma}\label{meanerroru_n}
    In the assumption of \cref{conforming posterior consistency}, we have
    \begin{equation*}
        ||u_n-u_n^\dagger||_{L^2(B)} = \mathcal{O}_{P_{D_N}}(\delta^\tau_N).
    \end{equation*}
\end{lemma}
\begin{proof}
    \textbf{Step 1.}We test (\ref{u_nequation}) with $v_n = u_n-u_n^\dagger$
    \begin{equation*}
    \begin{array}{cc}
    \langle u_n,u_n-u_n^\dagger\rangle_{I_N} + N\delta_N^2\langle u_n,u_n-u_n^\dagger\rangle_{\mathcal{C}_0^{-1}}= \langle u^\dagger,u_n-u_n^\dagger\rangle_{I_N}.
    \end{array}
\end{equation*}
Then we have
\begin{equation*}
    \begin{aligned}
        &||u_n-u_n^\dagger||_{I_N}^2+ N\delta_N^2||u_n-u_n^\dagger||_{\mathcal{C}_0^{-1}}^2 \\
        = &\langle u^\dagger-u_n^\dagger, u_n-u_n^\dagger\rangle_{I_N}+ N\delta_N^2\langle u^\dagger-u_n^\dagger, u_n-u_n^\dagger\rangle_{\mathcal{C}_0^{-1}}\\
        \le &\frac{1}{2}(||u_n-u_n^\dagger||_{I_N}^2+ N\delta_N^2||u_n-u_n^\dagger||_{\mathcal{C}_0^{-1}}^2)+\frac{1}{2}(||u^\dagger-u_n^\dagger||_{I_N}^2+N\delta_N^2||u^\dagger||_{\mathcal{C}_0^{-1}}^2).
    \end{aligned}
\end{equation*}
Therefore, dividing by $N$ on both sides, we have
\begin{equation*}
    \frac{1}{N}||u_n-u_n^\dagger||_{I_N}^2+ \delta_N^2||u_n-u_n^\dagger||_{\mathcal{C}_0^{-1}}^2\le \frac{1}{N}||u^\dagger-u_n^\dagger||_{I_N}^2+\delta_N^2||u^\dagger||_{\mathcal{C}_0^{-1}}^2
\end{equation*}
By the same idea as in the proof of \cref{montecarloconv}, we have
\begin{equation*}
    P_{D_N}(\frac{1}{N}||u^\dagger-u_n^\dagger||_{I_N}^2\ge C\delta_N^2||u^\dagger||^2_{H^\alpha(\Omega)})\lesssim \frac{1}{N\phi(n)^{2\alpha}}=N\delta_N^2\to 0.
\end{equation*}
Therefore, using the fact that $E\hookrightarrow H^\alpha(\Omega)$, there exists some $m>0$,
\begin{equation}\label{I_Nconverge}
\begin{array}{cc}
      P_{D_N}(||u_n-u^\dagger_n||^2_{H^\alpha(\Omega)}\le m^2)\to 1 \\
     P_{D_N}(\frac{1}{N}||u_n-u_n^\dagger||^2_{I_N}\le m^2\delta^2_N)\to 1.
\end{array}
\end{equation}
\textbf{Step 2.} We derive an auxiliary result that will be used to bound $||u_n-u_n^\dagger||_{L^2(\omega)}^2$. Let the set $B_N$ satisfy $ B_N:=\{v\in \mathcal{H}^\alpha(\Omega):||v||_{L^2(\omega)}\ge \bar{m}\delta_N, ||v||_{H^\alpha(\Omega)}\le m\}$, where $m$ is taken from (\ref{I_Nconverge}), and $\bar{m}>0$ is some large number to be chosen later. Since $v\in \mathcal{H}^\alpha(\Omega)$ is harmonic, there exists some $M>1$ such that
\begin{equation}\label{eq:LinftyboundedbyL2}
    \begin{array}{cc}
    ||v||_{L^\infty(\omega)} \le M||v||_{L^2(\Omega)}.  \\
    \end{array}
\end{equation}
Consider the metric entropy bound. By choosing large enough $\bar{m}$, we have
\begin{equation}\label{metric entropy}
    \log N(B_N,||\cdot||_{L^2(\Omega)},\frac{\bar{m}\delta_N}{4M}) \le N\delta_N^2,
\end{equation}
where $N(B_N,||\cdot||_{L^2(\omega)},\frac{\bar{m}\delta_N}{4M})$ is the least number of $L^2(\Omega)$ balls with radius $\frac{\bar{m}\delta_N}{4M}$ that covers $B_N$. The bound (\ref{metric entropy}) can be seen in, e.g., \cite[(A.19)]{nickl2023bayesian}. Now consider the set
\begin{equation*}
    A_N = \{X_N:\exists v\in B_N, \big | ||v||^2_{L^2(\omega)}-\frac{1}{N}||v||_{I_N}^2\big|\ge \frac{33}{64}||v||_{L^2(\omega)}\}.
\end{equation*}
From the previous metric entropy bound, we can find a net
\begin{equation*}
  V:=  \{v_i\}, i = \{1,2,...,N(B_N,||\cdot||_{L^2(\Omega)},\frac{\bar{m}\delta_N}{4M})\}
\end{equation*}
that covers $B_N$. For each $v_i$, we have by Bernstein's inequality and the fact that $||v_i||_{L^2(\omega)}\ge \bar{m}\delta_N$
\begin{equation*}
    P_{D_N}(\big| ||v_i||^2_{L^2(\omega)}-\frac{1}{N}||v_i||_{I_N}^2\big|\ge \frac{1}{16}||v_i||^2_{L^2(\omega)})\le \exp(-cN\delta_N^2),
\end{equation*}
where $c\propto \frac{\bar{m}}{m}$. Define 
\begin{equation*}
    A_N':= \{X_N:\exists v_k\in V, \big| ||v_k||^2_{L^2(\omega)}-\frac{1}{N}||v_k||_{I_N}^2\big|\ge\frac{1}{16}||v_k||_{L^2(\omega)}^2\}.
\end{equation*}
Note that $A_N'$ is Borel measurable since $V$ is finite dimensional. By (\ref{metric entropy}), we have
\begin{equation}\label{A_N'bound}
    P_{D_N}(A_N')\le \exp((1-c)N\delta_N^2).
\end{equation}
Now for any $v\in B_N$, there exists some $v_k\in V$, such that $||v-v_k||_{L^2(\Omega)}\le \frac{\bar{m}\delta_N}{4M}$.
\begin{equation}\label{I_Nbound}
    \frac{1}{N}||(v-v_k)||_{I_N}^2\le ||v-v_k||_{L^\infty(\omega)}^2\le M^2||v-v_k||^2_{L^2(\Omega)}\le \frac{1}{16}\bar{m}^2\delta^2_N.
\end{equation}
Therefore, when $\big | ||v_i||^2_{L^2(\omega)}-\frac{1}{N}||v_i||_{I_N}^2\big|\le \frac{1}{16}||v_i||^2_{L^2(\omega)}$ for all $v_i\in V$, we have $\forall v\in B_N$, there is some $v_k\in V$, such that
\begin{equation*}
\begin{aligned}
    &\Bigl| ||v||^2_{L^2(\omega)}-\frac{1}{N}||v||_{I_N}^2\Bigr|\\ \le &2\big| ||v_k||^2_{L^2(\omega)}-\frac{1}{N}||v_k||_{I_N}^2\big|+2\big| ||v-v_k||^2_{L^2(\omega)}-\frac{1}{N}||v-v_k||_{I_N}^2\big|\\
    \le &\frac{1}{8}||v_k||^2_{L^2(\omega)}+\frac{1}{4}\bar{m}^2\delta_N^2\\
    \le &\frac{1}{4}(||v||^2_{L^2(\omega)}+||v-v_k||^2_{L^2(\omega)})+\frac{1}{4}||v||_{L^2(\omega)}^2\\
    \le &\frac{33}{64}||v||^2_{L^2(\omega)},
\end{aligned}
\end{equation*}
where the second and the last inequalities uses (\ref{I_Nbound}) and the fact that $M>1$, and the third inequality uses the definition of $B_N$. This gives that $A_N\subset A_N'$. Then by choosing $\bar{m}$ large enough with respect to $m$ so that $c>1$ in (\ref{A_N'bound}), we have $P_{D_N}(A_N')\to 0$.\\
\textbf{Step 3.} We now use the above result and (\ref{I_Nconverge}) to bound $||u_n-u_n^\dagger||_{L^2(\omega)}^2$. Let $e_n = u_n-u_n^\dagger$. By (still) setting $\bar{m}$ large enough such that $\frac{31}{64}\bar{m}^2\ge m^2$, we have
\begin{equation*}
\begin{aligned}
    &P_{D_N}(||e_n||_{H^\alpha(\Omega)}\le m,||e_n||_{L^2(\omega)}^2\ge \bar{m}^2\delta_N^2)\\
    \le &P_{D_N}(||e_n||_{L^2(\omega)}^2\ge \bar{m}^2\delta_N^2, \big |||e_n||_{L^2(\omega)}^2-||e_n||_{I_N}\big|\le \frac{33}{64}||e_n||_{L^2(\omega)}^2)\\
    &+ P_{D_N}(||e_n||_{H^\alpha(\Omega)}\le m,||e_n||_{L^2(\omega)}^2\ge \bar{m}^2\delta_N^2, \big |||e_n||_{L^2(\omega)}^2-||e_n||_{I_N}\big|\ge \frac{33}{64}||e_n||_{L^2(\omega)}^2)\\
    \le &P_{D_N}(||e_n||_{I_N}^2\ge \frac{31}{64}\bar{m}^2\delta_N^2)+P_{D_N}(A_N')\to0,
\end{aligned}
\end{equation*}
where the last line uses (\ref{I_Nconverge}) and the result that $A_N\subset A_N'$ in step 2. Since $P_{D_N}(||e_n||_{H^\alpha(\Omega)}\le m)\to 1$, we have that $P_{D_N}(||e_n||_{L^2(\omega)}^2\ge \bar{m}^2\delta_N^2)\to 0$ and thus $P_{D_N}(||e_n||_{L^2(\omega)}^2\le \bar{m}^2\delta_N^2)\to 1$. Finally by using \cref{three-balls} we prove the lemma.
\end{proof}

\begin{lemma}\label{varianceu_n}
    Under the assumption of \cref{conforming posterior consistency}, we have
    \begin{equation}
        ||u_n^y-u_n||_{L^2(B)} = \mathcal{O}_{P_{D_N}}(\delta_N^\tau).
    \end{equation}
\end{lemma}
\begin{proof}
    The proof is identical to \cref{variancebound} and \cref{discrete posterior consistency}. Since $\phi(n) = \mathcal{O}((Nl_N)^{-\frac{1}{2\alpha+d}})$ and $\phi(n) = \mathcal{O}(n^{-\frac{1}{d}})$, we have $n = \mathcal{O}((Nl_N)^{\frac{d}{2\alpha+d}})$. Therefore, there exists some $C>0$, we have
    \begin{equation*}
        P_{D_N}(\frac{1}{N}||u_n^y-u_n||_{I_N}^2+ ||u_n^y-u_n||_{\mathcal{C}_0^{-1}}^2\ge \frac{Cn}{Nl_N})\to 0 
    \end{equation*}
    Now let $l_N = \frac{n}{N\delta_N^2}$. We have
     \begin{equation*}
        P_{D_N}(\frac{1}{N}||u_n^y-u_n||_{I_N}^2+ ||u_n^y-u_n||_{\mathcal{C}_0^{-1}}^2\ge C(Nl_N)^{-\frac{2\alpha}{2\alpha+d}})\to 0.
    \end{equation*}
    Finally by repeating the steps 2 and 3 in \cref{meanerroru_n}, we prove the lemma.
\end{proof}

\bibliographystyle{siamplain}
\bibliography{bibliography}
\end{document}